\documentclass[a4paper,11pt]{article}
\usepackage[utf8x]{inputenc}

\usepackage{a4wide}
\usepackage{amssymb}
\usepackage{amsmath}
\usepackage{amsthm}
\usepackage[mathscr]{euscript}
\usepackage[normalem]{ulem}
\usepackage[colorlinks=true]{hyperref}
\usepackage[dvipsnames]{xcolor}
\usepackage[title]{appendix}
\usepackage{enumerate}

\numberwithin{equation}{section}

\theoremstyle{definition}

\newtheorem{definition}{Definition}[section]
\newtheorem{remark}[definition]{Remark}

\theoremstyle{plain}

\newtheorem{theorem}[definition]{Theorem}
\newtheorem{proposition}[definition]{Proposition}
\newtheorem{lemma}[definition]{Lemma}
\newtheorem{corollary}[definition]{Corollary}

\newcommand{\ve}{\varepsilon}
\newcommand{\ez}{\epsilon}
\newcommand{\del}{\partial}
\newcommand{\lra}{\longrightarrow}
\newcommand{\e}{\mathrm{e}}
\newcommand{\dd}{\mathrm{d}}

\newcommand{\N}{\ensuremath{\mathbb{N}}}

\newcommand{\R}{\ensuremath{\mathbb{R}}}

\newcommand{\cE}{\ensuremath{\mathcal{E}}}
\newcommand{\cT}{\ensuremath{\mathcal{T}}}
\newcommand{\fm}{\ensuremath{\mathfrak{m}}}
\newcommand{\fn}{\ensuremath{\mathfrak{n}}}
\newcommand{\bb}{\ensuremath{\mathbf{b}}}
\newcommand{\bL}{\ensuremath{\mathbf{L}}}

\newcommand{\LL}{\ensuremath{\mathscr{L}}}

\def\loc{\mathop{\mathrm{loc}}\nolimits}
\def\vol{\mathop{\mathrm{vol}}\nolimits}
\def\div{\mathop{\mathrm{div}}\nolimits}

\def\Ric{\mathop{\mathrm{Ric}}\nolimits}
\def\trace{\mathop{\mathrm{trace}}\nolimits}
\def\HS{\mathop{\mathrm{HS}}\nolimits}

\newcommand{\rev}[1]{\overleftarrow{#1}}
\newcommand{\ol}[1]{\overline{#1}}

\title{Splitting theorems for weighted Finsler spacetimes\\ via the $p$-d'Alembertian: beyond the Berwald case}

\author{Erasmo CAPONIO\thanks{
Dipartimento di Meccanica, Matematica e Management,
Politecnico di Bari, via Orabona 4 - 70125 Bari, Italy
(\textsf{erasmo.caponio@poliba.it})}
\and
Argam OHANYAN\thanks{
Faculty of Mathematics, University of Vienna,
Oskar-Morgenstern-Platz 1, 1090 Vienna, Austria
(\textsf{argam.ohanyan@univie.ac.at})}
\and
Shin-ichi OHTA\thanks{
Department of Mathematics, Osaka University, Osaka 560-0043, Japan
(\textsf{s.ohta@math.sci.osaka-u.ac.jp});
RIKEN Center for Advanced Intelligence Project (AIP),
1-4-1 Nihonbashi, Tokyo 103-0027, Japan}}

\begin{document}

\maketitle

\begin{abstract}
	A timelike splitting theorem for Finsler spacetimes was previously established by the third author, in collaboration with Lu and Minguzzi, under relatively strong hypotheses, including the Berwald condition. This contrasts with the more general results known for positive definite Finsler manifolds.
	In this article, we employ a recently developed strategy for proving timelike splitting theorems using the elliptic $p$-d'Alembertian. This approach, pioneered by Braun, Gigli, McCann, Sämann, and the second author, allows us to remove the restrictive assumptions of the earlier splitting theorem.
		For timelike geodesically complete Finsler spacetimes, we establish a diffeomorphic splitting. In the specific case of Berwald spacetimes, we show that the Busemann function generates a group of isometries via translations. Furthermore, for Berwald spacetimes, we extend these splitting theorems by replacing the assumption of timelike geodesic completeness with global hyperbolicity.
		Our results encompass and generalize the timelike splitting theorems for weighted Lorentzian manifolds previously obtained by Case, Woolgar and Wylie.
\end{abstract}         
\vspace{1em}

\medskip

\tableofcontents

\section{Introduction}\label{sc:intro}

Splitting theorems are important rigidity results which appear in various fields of geometry. They state that a complete geometric space of nonnegative curvature containing an infinitely long, distance-realizing curve (i.e., a line) must split off that curve as an isometric factor of the real line. The intuition is that nonnegative curvature promotes the formation of conjugate points and thus curves should stop realizing the distance in finite time, however, the existence of a line is an anomaly that is only possible in product geometries.

In both Riemannian and Lorentzian geometry, splitting theorems under the assumption of nonnegative Ricci curvature are major milestones of their respective fields.
In the Riemannian case, Cheeger--Gromoll's celebrated splitting theorem \cite{CG71} and its generalizations (e.g., the almost-splitting of Cheeger--Colding \cite{CheegerColding96} or the generalization to $\mathsf{RCD}$-spaces by Gigli \cite{gigli2013}) are important building blocks for the theory of geometric structures of (possibly singular) spaces (see Cheeger--Colding \cite{CC97}, Mondino--Naber \cite{MondinoNaber2019}).
In the context of general relativity, the nonnegativity of the Ricci curvature in timelike directions is usually called the timelike convergence condition (which, under the Einstein equations, is implied by the strong energy condition).  The splitting of a timelike geodesically complete spacetime under the strong energy condition  was conjectured by Yau \cite{Yau82} and established by Newman \cite{Newman:90}, building on earlier work of Eschenburg \cite{Eschenburg:1988} and on the globally hyperbolic case proved by Galloway \cite{Galloway:89}.  Consequently, a  dichotomy arises between the product splitting of a spacetime and the existence of incomplete timelike  geodesics (see the standard texts of Hawking--Ellis \cite{HawkingEllisAnniv2023}, Beem--Ehrlich--Easley \cite{BEE:1996}).
These results are preceded by a splitting theorem under the stronger assumption of a timelike sectional curvature bound due to Beem--Ehrlich--Markvorsen--Galloway \cite{BeemEhrlichMarkvorsenGalloway84, BEMG85} (this latter result has recently been generalized to the synthetic setting of Lorentzian length spaces by Kunzinger--Sämann \cite{KS:18}; see  \cite{BeranOhanyanRottSolis23}).

Both the Riemannian and the Lorentzian splitting theorems under nonnegative Ricci curvature bounds were generalized to the Finsler setting (see Ohta \cite{Osplit} in the Riemann--Finsler case, Lu--Minguzzi--Ohta \cite{LMO3} in the Lorentz--Finsler case), adopting appropriate weighted formulations.
In contrast to the results for Finsler manifolds in \cite{Osplit}, the work on Finsler spacetimes in \cite{LMO3} required stronger hypotheses. More precisely, in addition to the usual set of assumptions (timelike geodesic completeness, nonnegative weighted Ricci curvature in timelike directions, existence of a timelike line), it is  assumed that there exists a timelike geodesically complete Lorentzian metric $g$ on $M$ whose Levi-Civita connection agrees with the Chern connection (thus, the Lorentz--Finsler spacetime $(M,L)$ has  to be  of {\em Berwald type})  and  such that  the given timelike line for the Lorentz--Finsler metric lifts as a timelike line in the  universal cover of $(M,g)$. 

The aim of this article is to improve the splitting theorem in \cite[Theorem~1.2]{LMO3} by removing these technical assumptions. Precisely, although the isometric splitting is clearly out of reach (see Remark~\ref{rm:split}), we can obtain a diffeomorphic, measure-preserving splitting and, under the  additional  Berwald condition, also obtain a one-parameter family of isometries as translations (a local Finslerian isometric splitting can be obtained under significantly strong symmetry assumptions on $L$, see Caponio--Stancarone \cite[Theorem 4.8] {CapSta18}, that must also lack twice-differentiability along a timelike direction; see also \cite[Remark 7.5]{CapCor23}).
In particular, while \cite{LMO3} treated only Newman's timelike geodesic completeness version \cite{Newman:90}, we also prove a counterpart of Galloway's global hyperbolicity theorem \cite{Galloway:89} in the Berwald setting.

Such an improvement stems from a new strategy for Lorentzian manifolds developed by Braun--Gigli--McCann--Ohanyan--Sämann \cite{quintet}, based on the $p$-d'Alembertian, a quasi-linear second order differential operator   defined analogously to the $p$-Laplacian in the positive definite case, but with $p < 1$. The $p$-d'Alembertian was introduced by Mondino--Suhr \cite{MS:22} and later investigated in the non-smooth setting by Beran--Braun--Calisti--Gigli--McCann--Ohanyan--Rott--Sämann \cite{octet}. In the latter work, its connection with a convex $p$-energy functional was used to reveal its elliptic nature in metric measure spacetimes.  This ellipticity was then used in \cite{quintet} to simplify proofs of the Lorentzian splitting theorems. 
Unlike the standard ($2$-)d'Alembertian (the spacetime Laplacian), which is hyperbolic and requires identifying a suitable spacelike hypersurface for maximum principle arguments, the  ellipticity of the $p$-d'Alembertian allows for a direct application of the maximum principle, yielding a simpler, alternative proof of the timelike splitting theorem in both timelike geodesically complete and globally hyperbolic spacetimes (see \cite{quintet} for details). The same set of authors also extended this approach to weighted Lorentzian manifolds with non-smooth metrics and weights (see \cite{quintet2}).

We extend the elliptic \(p\)-d'Alembertian strategy to smooth Finsler spacetimes.
In this setting the operator is  anisotropic, and even the case \(p=2\) yields a quasilinear d'Alembertian.
Accordingly, the point is not to ``trade linearity for ellipticity'' \cite{mccann2025}, but rather to obtain an anisotropic elliptic operator within the Finsler spacetime framework.

A Finsler spacetime $(M,L)$ is a time oriented Lorentz--Finsler manifold; see Definition~\ref{df:spacetime}. Lorentz--Finsler structures and Finsler spacetimes have been considered both as a 
geometric description of gravity beyond general relativity (see, e.g., \cite{KoStSt09, Pfeife19, BeJaSa20}) and in the classical limit of modified dispersion relations
encompassing Lorentz violation in quantum gravity and in the Standard Model Extension (see, e.g., \cite{GiLiSi06, Vacaru11, Kostel11, KoRuTs12, Russel15, Collad17}). Remarkably, they provide a counter-example to the Schiff conjecture \cite{LaPeHa12} but they are compatible with a slight modification of Ehlers--Pirani--Schild axioms for general relativity (see \cite{LaePer18, BeJaSa20}).

We denote by $\tau$ the time separation function on $M$ (see \eqref{timesep}).
Then, given a timelike straight line $\eta\colon \R \lra M$ (i.e., $\tau(\eta(s),\eta(t))=t-s$ for all $s<t$), the associated Busemann function $\bb_\eta \colon M \lra \R$ defined by
\[ \bb_\eta(x) :=\lim_{t \to \infty} \bigl\{ t-\tau\bigl( x,\eta(t) \bigr) \bigr\} \]
plays a vital role to establish a splitting theorem along $\eta$.
Our first main theorem is the following.

\begin{theorem}[Splitting theorem]\label{th:LFsplit}
Let $(M,L)$ be a connected Finsler spacetime equipped with a measure $\fm$ on $M$.
Suppose the following:
\begin{enumerate}[{\rm (1)}]
\item\label{it:split1}
$(M,L)$ is either timelike geodesically complete, or Berwald and globally hyperbolic;
\item\label{it:split2}
there is a complete timelike straight line $\eta\colon \R \lra M$;
\item\label{it:split3}
$(M,L,\fm)$ satisfies $\Ric_N \ge 0$ in timelike directions for some $N \in (-\infty,0) \cup [n,\infty]$, where $n=\dim M$;
\item\label{it:split4}
in the case of $N \in (-\infty,0) \cup \{\infty\}$,
$(M,L)$ satisfies the timelike $\ez$-completeness for some $\ez$ in the range \eqref{eq:erange} associated with $N$.
\end{enumerate}
Then, the Lorentzian manifold $(M,g_{\nabla(-\bb_{\eta})})$ isometrically splits
in the sense that there exists an $(n-1)$-dimensional Riemannian manifold $(\Sigma,h)$
such that $(M,g_{\nabla(-\bb_{\eta})})$ is isometric to $(\R \times \Sigma,-\dd t^2 +h)$.
Moreover, $(\Sigma,h)$ is equipped with a measure $\fn$ for which
$\fm$ coincides with the product of the Lebesgue measure $\dd t$ on $\R$ and $\fn$.
\end{theorem}

The Lorentzian structure $g_{\nabla(-\bb_\eta)}$ is the second order approximation of $L$ in the direction $\nabla(-\bb_\eta)$ (see \eqref{eq:g_v}).
One can also see that, for every $x \in \Sigma$, the curve $\zeta(t)$ corresponding to $t \longmapsto (t,x)$ is a timelike straight line bi-asymptotic to $\eta$ and an integral curve of the gradient vector field $\nabla(-\bb_\eta)$.
We remark that $\Sigma$ is given as the level surface $\bb_\eta^{-1}(0)$, and then $(\Sigma,L|_{T\Sigma})$ is spacelike if $L$ is reversible and $\dim M \ge 3$, though it is unclear whether $L|_{T\Sigma}$ is positive definite (strongly convex); see Corollary~\ref{cr:Sigma} for details.

The weighted Ricci curvature $\Ric_N$ is a modification of the usual Ricci curvature taking into account the behavior of the measure $\fm$, and timelike $\ez$-completeness can be regarded as a weighted version of timelike geodesic completeness (see Definitions~\ref{df:wRic}, \ref{df:cmplt}).
When $N \in [n,\infty)$, we can take $\ez=1$ and timelike $1$-completeness always holds true.

The Berwald condition in \eqref{it:split1} means that the covariant derivative is independent of a reference vector (see Definition~\ref{df:Ber}).
It allows us to show that, in the globally hyperbolic setting, nearby asymptotes to the straight line $\eta$ are complete, and that the local splitting around $\eta$ can be globalized, along the lines of Galloway's argument given in \cite{Galloway:89}.
Since such difficulties do not arise in the timelike geodesically complete setting, we do not need to assume the Berwald condition there.
Let us remark that, in Lu--Minguzzi--Ohta \cite[Theorem~1.2]{LMO3}, only timelike geodesically complete Berwald spacetimes which are Lorentzian-metrizable were considered (see also Remark~\ref{rm:met}).

\begin{remark}\label{rm:split}
Theorem~\ref{th:LFsplit} implies the diffeomorphic and measure-preserving splitting of $(M,\fm)$, however, the metric $L$ does not have any splitting structure.
This is actually natural in the Lorentz--Finsler framework, for one can modify $L$ in spacelike directions without changing it in timelike directions.  Since the hypotheses are concerned only with (future-directed) timelike directions, such a deformation in spacelike directions does not influence them.
Hence, in the situation of Theorem~\ref{th:LFsplit}, we have no control of $L$ in spacelike directions.
\end{remark}

In the Berwald case, thanks to its more rigid structure, we have in fact a certain control of $L$ also in spacelike directions, in the same spirit as Ohta \cite[Proposition~5.2]{Osplit} and Lu--Minguzzi--Ohta \cite[Corollary~1.3]{LMO3}.

\begin{theorem}[Isometric translations in Berwald spacetimes]\label{th:Berwald}
Let $(M,L,\fm)$ be a Berwald spacetime satisfying the hypotheses in Theorem~$\ref{th:LFsplit}$.
Then, we have the following.
\begin{enumerate}[{\rm (i)}]
\item\label{it:isom}
In the product structure $M=\R \times \Sigma$, the translations
\[ \Phi_t(s,x):=(s+t,x), \qquad (s,x) \in \R \times \Sigma,\,\ t \in \R, \]
are isometric transformations of $(M,L)$ and preserve the measure $\fm$.

\item\label{it:geod}
The geodesic equation of $M$ splits into those of $\R$ and $\Sigma$.
Precisely, a curve in $M$ is geodesic if and only if its projections to $\R$ and $\Sigma$ are geodesic.
\end{enumerate}
\end{theorem}

As is explained in Remark~\ref{rm:split}, one cannot expect the isometry of translations in general Finsler spacetimes.
We also remark that, since we consider measured Finsler spacetimes and there may be no canonical measure (like the volume measure in the Lorentzian case) in the Lorentz--Finsler setting, Theorem~\ref{th:LFsplit} does not cover the unweighted case (assuming $\Ric \ge 0$ in place of $\Ric_N \ge 0$).
One can define the $p$-d'Alembertian without using a measure (Remark~\ref{rm:unwght}), however, what is missing is a linear divergence, which appears in a Bochner-type identity (Proposition~\ref{pr:Boch}) used to obtain the smoothness of the Busemann function.

In the Lorentzian case, the weighted version studied by Case \cite[Theorem~1.2]{Ca} ($N \in [n,\infty]$) and Woolgar--Wylie \cite[Theorem~1.5]{WW} ($N \in (-\infty,1) \cup [n,\infty]$) includes the unweighted one of Galloway \cite{Galloway:89} and Newman \cite{Newman:90} by choosing the volume measure as $\fm$ and letting $N=n$.
For a weighted Lorentzian manifold $(M,g,\fm)$, since $g_{\nabla(-\bb_\eta)}$ coincides with $g$, Theorem~\ref{th:LFsplit} recovers the splitting theorems in \cite{Ca,WW} as the special case of $\ez=1$ (resp.\ $\ez=0$) for $N \in [n,\infty)$ (resp.\ $N \in (-\infty,0) \cup \{\infty\}$) (the case of $N \in [0,1)$ was excluded for a technical reason; see Remark~\ref{rm:Wylie}).

The article is organized as follows.
After reviewing necessary notions and known results on weighted Finsler spacetimes in Section~\ref{sc:pre}, we introduce the $p$-d'Alembertian in Section~\ref{sc:p-dA}.
Then we discuss the properties of Busemann functions in Section~\ref{sc:busemann}. Section~\ref{sc:harm} includes the main ingredients: the $p$-harmonicity of the Busemann function based on the maximum principle via the $p$-d'Alembertian, and the Bochner-type identity (for improving the regularity of the Busemann function).
In the latter step, in contrast with the Bochner-type identity for the $p$-d'Alembertian in \cite{quintet}, we employ the identity in terms of the usual d'Alembertian, which turns out to be sufficient for the splitting theorem.
Finally, Section~\ref{sc:split} is devoted to the proofs of Theorems~\ref{th:LFsplit}, \ref{th:Berwald}.

\section{Preliminaries for Finsler spacetimes}\label{sc:pre}
We briefly recall the necessary notions for Lorentz--Finsler manifolds along the lines of  \cite{LMO1,LMO2,LMO3} 
(we remark that $\dim M=n+1$ there).
We refer to  \cite{O'Neill:1983} and  \cite{BEE:1996} for the basics of Lorentzian geometry and to \cite{Perlic06,Min-cone, Min-Rev, Min-causality, JavSan20} for some generalizations including Lorentz--Finsler manifolds.

Throughout the article,
let $M$ be a connected $C^{\infty}$-manifold without boundary of dimension $n\ge 2$.
Given a local coordinate system $(x^i)_{i=1}^n$ on an open set $U\subset M$,
we will use the fiber-wise linear coordinates
$v =\sum_{i=1}^n v^i (\partial/\partial x^i)|_x,$ $x \in U$, on the coordinate patch $TU$ in the tangent bundle $TM$.

\subsection{Finsler spacetimes}\label{ssc:Finsler}
For our purposes, we adopt the following notion of Lorentz--Finsler structure:
\begin{definition}[Lorentz--Finsler structures]\label{df:LFstr}
A \emph{Lorentz--Finsler structure} on $M$ is a function
$L\colon TM \lra \R$ satisfying the following conditions:
\begin{enumerate}[(1)]
\item $L \in C^{\infty}(TM \setminus \{0\})$;
\item $L(cv)=c^2 L(v)$ for all $v \in TM$ and $c>0$;
\item For any $v \in TM \setminus \{0\}$, the symmetric matrix
\begin{equation}\label{eq:g_ij}
\big( g_{ij}(v) \big)_{i,j=1}^n
 :=\bigg( \frac{\del^2 L}{\del v^i \del v^j}(v) \bigg)_{i,j=1}^n
\end{equation}
is non-degenerate with signature $(-,+,\ldots,+)$.
\end{enumerate}
Then we call $(M,L)$ a ($C^{\infty}$-)\emph{Lorentz--Finsler manifold}.
\end{definition}

For $v \in T_xM \setminus \{0\}$, we can define a Lorentzian metric $g_v$ on $T_xM$
by using \eqref{eq:g_ij} as
\begin{equation}\label{eq:g_v}
g_v \Bigg( \sum_{i=1}^n a_i \frac{\del}{\del x^i}\bigg|_x,
 \sum_{j=1}^n b_j \frac{\del}{\del x^j}\bigg|_x \Bigg)
 :=\sum_{i,j=1}^n g_{ij}(v) a_i b_j.
\end{equation}
Then we have $g_v(v,v)=2L(v)$ by Euler's homogeneous function theorem.

A tangent vector $v \in TM $ is said to be \emph{timelike} (resp.\ \emph{null})
if $L(v)<0$ (resp.\ $L(v)=0$).
We say that $v$ is \emph{lightlike} if it is null and nonzero, and \emph{causal} if it is timelike or lightlike.
\emph{Spacelike} vectors are those for which $L(v)>0$ or $v=0$.
Denote by $\Omega'_x \subset T_xM$ the set of timelike vectors and put
$\Omega' :=\bigcup_{x \in M} \Omega'_x$.
For causal vectors $v$, we define
\begin{equation}\label{eq:LtoF}
F(v) :=\sqrt{-2L(v)} =\sqrt{-g_v(v,v)}.
\end{equation}
Note that $\Omega'_x \neq \emptyset$, every connected component of $\Omega'_x$ is a convex cone (see \cite{Be},  \cite[Lemma~2.3]{LMO1}),
and that the closures of different components intersect only at $0$ (see  \cite[Proposition~1]{Min-cone}).

\begin{remark}\label{rm:Omega}
On the one hand, the number of connected components of $\Omega'_x$ may be larger than $2$ in this generality (see \cite{Be}, \cite[Example~2.4]{LMO1}).
On the other hand, if $L$ is reversible (i.e., $L(v)=L(-v)$ for any $v\in TM$) and $n \ge 3$, then $\Omega'_x$ has exactly two connected components (see \cite[Theorem 7]{Min-cone}).
\end{remark}

\begin{definition}[Finsler spacetimes]\label{df:spacetime}
If a Lorentz--Finsler manifold $(M,L)$ admits a smooth timelike vector field $X$,
then $(M,L)$ is said to be \emph{time oriented} (by $X$).
A time oriented Lorentz--Finsler manifold will be called a \emph{Finsler spacetime}.
\end{definition}

In a Finsler spacetime time oriented by $X$,
a causal vector $v \in T_xM$ is said to be \emph{future-directed}
if it lies in the same connected component of $\overline{\Omega'}\!_x \setminus \{0\}$ as $X(x)$.
We denote by $\Omega_x \subset \Omega'_x$ the set of future-directed timelike vectors,
and define
\[ \Omega :=\bigcup_{x \in M} \Omega_x, \qquad
 \overline{\Omega} :=\bigcup_{x \in M} \overline{\Omega}_x. \]
A $C^1$-curve in $(M,L)$ is said to be \emph{timelike} (resp.\ \emph{causal})
if its tangent vector is always timelike (resp.\ causal).
Henceforth, unless explicitly stated otherwise, all causal curves are assumed to be future-directed.

Given $x,y \in M$, we write $x \ll y$ (resp.\ $x<y$)
if there is a timelike (resp.\ causal) curve from $x$ to $y$,
and $x \le y$ means that $x=y$ or $x<y$.
Then we define the \emph{chronological past} and \emph{future} of $x$ by
\[ I^-(x):=\{y \in M \,|\, y \ll x\}, \qquad I^+(x):=\{y \in M \,|\, x \ll y\}, \]
and the \emph{causal past} and \emph{future} of $x$ by
\[ J^-(x):=\{y \in M \,|\, y \le x\}, \qquad J^+(x):=\{y \in M \,|\, x \le y\}. \]

We also define the \emph{time separation} (also called the \emph{Lorentz--Finsler distance}) $\tau(x,y)$ for $x,y \in M$ by
\begin{equation}\label{timesep} \tau(x,y) :=\sup_{\eta} \bL(\eta), \qquad
 \bL(\eta) :=\int_0^1 F\big( \dot{\eta}(t) \big) \,\dd t, \end{equation}
where $\eta\colon [0,1] \lra M$ runs over all causal curves from $x$ to $y$. We set $\tau(x,y):= 0$ if $x \not \leq y$.
A curve $\eta$ attaining the above supremum and having a constant speed (i.e., $F(\dot\eta)$ is constant), which is then a causal \emph{geodesic}, is said to be \emph{maximizing}.
In general, $\tau$ is only lower semi-continuous and can be infinite (see, e.g.,  \cite[Proposition~6.7]{Min-Ray}).
In \emph{globally hyperbolic} Finsler spacetimes (see Definition~\ref{df:gh}), $\tau$ is finite and continuous,
and any pair of points $x,y \in M$ with $x<y$ admits a maximizing geodesic from $x$ to $y$
(see \cite[Propositions~6.8, 6.9]{Min-Ray}).

Next, we introduce the covariant derivative.
Define
\[ \gamma^i_{jk} (v)
 :=\frac{1}{2} \sum_{l=1}^n g^{il}(v)
 \bigg\{ \frac{\del g_{lk}}{\del x^j}(v) +\frac{\del g_{jl}}{\del x^k}(v)
 -\frac{\del g_{jk}}{\del x^l}(v) \bigg\} \]
for $i,j,k =1,\ldots,n$ and $v \in TM\setminus \{0\}$,
where $(g^{ij}(v))$ is the inverse matrix of $(g_{ij}(v))$;
\[ G^i(v) :=\sum_{j,k=1}^n \gamma^i_{jk}(v) v^j v^k, \qquad
N^i_j(v) :=\frac{1}{2}\frac{\del G^i}{\del v^j}(v) \]
for $v \in TM \setminus \{0\}$, $G^i(0)=N^i_j(0):=0$ by convention.
(The above $G^i(v)$ is defined in the same way as  \cite{Obook} and \cite{OSbw} and corresponds to $2G^\alpha(v)$ in  \cite{LMO3}, in order to avoid confusion in the calculations in Subsection~\ref{ssc:Bochner}.)
For later use, note that the homogeneous function theorem yields (cf.\ \cite[Exercise~4.7]{Obook})
\begin{equation}\label{eq:N}
N^i_j(v) =\sum_{k=1}^n \gamma^i_{jk}(v)v^k -\frac{1}{2} \sum_{k,l=1}^n g^{ik}(v) \frac{\del g_{kl}}{\del v^j}(v) G^l(v).
\end{equation}
Moreover, we set
\begin{equation}\label{eq:Gamma}
\Gamma^i_{jk}(v):=\gamma^i_{jk}(v)
 -\frac{1}{2}\sum_{l,m=1}^n g^{il}(v)
 \bigg( \frac{\del g_{lk}}{\del v^m}N^m_j
 +\frac{\del g_{jl}}{\del v^m} N^m_k
 -\frac{\del g_{jk}}{\del v^m} N^m_l \bigg)(v)
\end{equation}
on $TM \setminus \{0\}$,
and the \emph{covariant derivative} of a vector field $X=\sum_{i=1}^n X^i (\del/\del x^i)$ is defined as
\[ D_v^w X :=\sum_{i,j=1}^n
 \Bigg\{ v^j \frac{\del X^i}{\del x^j}(x)
 +\sum_{k=1}^n \Gamma^i_{jk}(w) v^j X^k(x) \Bigg\}
 \frac{\del}{\del x^i} \bigg|_x \]
for $v \in T_xM$ with a \emph{reference vector} $w \in T_xM \setminus \{0\}$.
We remark that the functions $\Gamma^i_{jk}$ in \eqref{eq:Gamma} are the coefficients of the \emph{Chern}(--\emph{Rund}) \emph{connection}.

In the Lorentzian case, $g_{ij}$ is constant in each tangent space (thus, $\Gamma^i_{jk}=\gamma^i_{jk}$) and the covariant derivative does not depend on the choice of a reference vector.
In the Lorentz--Finsler setting, the following class is worth considering.

\begin{definition}[Berwald spacetimes]\label{df:Ber}
A Finsler spacetime $(M,L)$ is said to be of \emph{Berwald type} (or called a \emph{Berwald spacetime}) if $\Gamma^i_{jk}$ is constant on the slit tangent space $T_xM \setminus \{0\}$ for any $x$ in the domain of every local coordinate system.
\end{definition}

By definition, the covariant derivative on a Berwald spacetime is defined independently of the choice of a reference vector.
An important property of a Berwald spacetime is that, for any $C^1$-curve $\eta\colon [0,1] \lra M$ whose velocity does not vanish, the parallel transport along $\eta$ gives a linear isometry from $(T_{\eta(0)}M,L)$ to $(T_{\eta(1)}M,L)$ (see, e.g.,  \cite[Proposition~6.5]{Obook} for the positive definite case).
In particular, all tangent spaces are mutually linearly isometric.
Examples of Berwald spacetimes include Lorentzian manifolds
and flat Lorentz--Finsler structures of $\R^n$
(every tangent space $T_x\R^n$ is canonically isometric to $T_0\R^n$ and $\Gamma^i_{jk}=\gamma^i_{jk}=0$).
We refer to \cite{FP,GT,FPP,CapMas20,FHPV,HPV} for some mathematical and physical investigations on Berwald spacetimes,
and to \cite[Chapter~10]{BCS}, \cite[\S\S 6.3, 10.2]{Obook} for the positive definite case.

\begin{remark}[Metrizability]\label{rm:met}
In the positive definite case,
Szab\'o showed that a Finsler manifold of Berwald type $(M,F)$ admits a Riemannian metric $h$ whose Levi-Civita connection coincides with the Chern connection of $F$,
i.e., the Christoffel symbols of $h$ coincide with $\Gamma^i_{jk}$ of $F$
(see  \cite{Sz} and  \cite[Exercise~10.1.4]{BCS}).
This is called the (\emph{Riemannian}) \emph{metrizability theorem}.
It is not known whether the metrizability can be generalized to Berwald spacetimes.
In \cite{FHPV}, some counter-examples were constructed for Lorentz--Finsler structures defined only on a subset of $TM$.
Their discussion is not applicable to Lorentz--Finsler structures defined on all of $TM$ as in Definition~\ref{df:LFstr}.
\end{remark}

The \emph{geodesic equation} is written as $D^{\dot{\eta}}_{\dot{\eta}}\dot{\eta} \equiv 0$.
The \emph{exponential map} is defined in the same way as the Riemannian case,
which is $C^{\infty}$ on a neighborhood of the zero section only in Berwald spacetimes
(see, e.g., \cite{Min-coord}, as well as
\cite[Exercise~5.3.5]{BCS} in the positive definite case).
For $C^1$-vector fields $X,Y$ along a causal geodesic $\eta$, we have
\begin{equation}\label{eq:g_eta}
\frac{\dd}{\dd t}\bigl[ g_{\dot{\eta}}(X,Y) \bigr] =g_{\dot{\eta}}(D^{\dot{\eta}}_{\dot{\eta}}X,Y) +g_{\dot{\eta}}(X,D^{\dot{\eta}}_{\dot{\eta}}Y)
\end{equation}
and, if $X$ is nowhere vanishing,
\begin{equation}\label{eq:g_X}
\frac{\dd}{\dd t}\bigl[ g_X(X,Y) \bigr] =g_X(D^X_{\dot{\eta}}X,Y) +g_X(X,D^X_{\dot{\eta}}Y)
\end{equation}
(see, e.g.,  \cite[(3.1), (3.2)]{LMO1}).

A $C^{\infty}$-vector field $J$ along a geodesic $\eta$ is called a \emph{Jacobi field}
if it satisfies the \emph{Jacobi equation} $D^{\dot{\eta}}_{\dot{\eta}} D^{\dot{\eta}}_{\dot{\eta}} J +R_{\dot{\eta}}(J) =0$,
where
\[ R_v(w):=\sum_{i,j=1}^n R^i_j(v) w^j \frac{\del}{\del x^i} \bigg|_x \]
for $v,w \in T_xM$ and
\[ R^i_j(v) :=\frac{\del G^i}{\del x^j}(v)
 -\sum_{k=1}^n \bigg\{ \frac{\del N^i_j}{\del x^k}(v) v^k
 -\frac{\del N^i_j}{\del v^k}(v) G^k(v) +N^i_k(v) N^k_j(v) \bigg\} \]
is the \emph{curvature tensor}.
A Jacobi field is also characterized as the variational vector field of a geodesic variation.

Note that $R_v(w)$ is linear in $w$, thus $R_v \colon T_xM \lra T_xM$ is an endomorphism
for each $v \in T_xM$.
For $v \in \overline{\Omega}_x$, we define the \emph{Ricci curvature} of $v$ as the trace of $R_v$: $\Ric(v):=\trace(R_v)$.
We remark that $\Ric(cv)=c^2 \Ric(v)$ for $c>0$.

\subsection{Legendre transforms and differential operators}\label{ssc:Lap}

In order to introduce the d'Alembertian, we consider the dual structure to $L$ and the Legendre transform
(see  \cite{Min-cone}, \cite[\S 3.1]{Min-causality},  \cite[\S 4.4]{LMO2} for further discussions, and  \cite[\S 3.2]{Obook} for the positive definite case).
Define the \emph{polar cone} to $\Omega_x \subset T_xM$ by
\[ \Omega^*_x :=\big\{ \omega \in T_x^*M \mid
 \omega(v)<0\ \text{for all}\ v \in \overline{\Omega}_x \setminus \{0\} \big\}. \]
This is an open convex cone in $T_x^*M$.
For $\omega \in \ol{\Omega}{}^*_x$, we define
\[ L^*(\omega)
 := -\frac{1}{2} \biggl( \sup_{v \in \Omega_x \cap F^{-1}(1)} \omega(v) \biggr)^2
 = -\frac{1}{2} \inf_{v \in \Omega_x \cap F^{-1}(1)} \bigl( \omega(v) \bigr)^2, \]
in other words, setting $F^*(\omega):=\sqrt{-2L^*(\omega)}$ for $\omega \in \ol{\Omega}{}^*$,
\[ F^*(\omega) = -\sup_{v \in \Omega_x \cap F^{-1}(1)} \omega(v). \]
Then we have the \emph{reverse Cauchy--Schwarz inequality}
\begin{equation}\label{eq:rCS}
\sqrt{4L^*(\omega) L(v)} =F^*(\omega)F(v) \le -\omega(v)
\end{equation}
for any $v \in \ol{\Omega}_x$ and $\omega \in \ol{\Omega}{}^*_x$.

\begin{definition}[Legendre transform]\label{df:Leg}
Define the \emph{Legendre transform}
$\LL^* \colon \Omega^*_x \lra \Omega_x$ as the map sending $\omega \in \Omega^*_x$
to the unique element $v \in \Omega_x$ satisfying $L(v)=L^*(\omega)=\omega(v)/2$. 
\end{definition}

We can continuously extend to $\LL^* \colon \ol{\Omega}{}^*_x \lra \ol{\Omega}_x$ (then $\LL^*(0)=0$), and equality holds in \eqref{eq:rCS} if and only if $v$ and $\LL^*(\omega)$ are proportional (see \cite[Theorem~3]{Min-cone}).

A coordinate expression of the Legendre transform is given by
\begin{equation}\label{eq:Leg}
\LL^*(\omega)
 =\sum_{i=1}^n \frac{\del L^*}{\del \omega_i}(\omega) \frac{\del}{\del x^i}
 =-\frac{1}{2} \sum_{i=1}^n \frac{\del[(F^*)^2]}{\del \omega_i}(\omega) \frac{\del}{\del x^i},
\end{equation}
where $\omega =\sum_{i=1}^n \omega_i \,\dd x^i$.
We set
\[ g^*_{ij}(\omega) :=\frac{\del^2 L^*}{\del \omega_i \del \omega_j}(\omega) \]
and note that, by the homogeneous function theorem,
\begin{equation}\label{eq:Euler}
\sum_{k=1}^n \frac{\del g^*_{ij}}{\del \omega_k}(\omega) \omega_k
=\sum_{k=1}^n \frac{\del g^*_{ik}}{\del \omega_j}(\omega) \omega_k
=0
\end{equation}
for all $i,j=1,\ldots,n$.
We also remark that $\LL^*\colon \Omega^*_x \lra \Omega_x$ is a bijection and $(g^*_{ij}(\omega))$ is the inverse matrix of $(g_{ij}(\LL^*(\omega)))$ (see \cite[Corollary~4]{Min-cone}).

A continuous function $f\colon M \lra \R$ is called a \emph{time function}
if $f(x)<f(y)$ for all $x,y \in M$ with $x<y$.
A $C^1$-function $f\colon M \lra \R$ is said to be \emph{temporal} if $-\dd f(x) \in \Omega^*_x$ for all $x \in M$.
Note that a temporal function is a time function.

For a temporal function $f\colon M \lra \R$,
define the \emph{gradient vector} of $-f$ at $x \in M$ by
\begin{equation}\label{eq:grad}
\nabla(-f)(x) :=\LL^* \bigl( -\dd f(x) \bigr)
 =\sum_{i,j=1}^n g^*_{ij}\bigl( -\dd f(x) \bigr)
 \frac{\del(-f)}{\del x^j}(x) \frac{\del}{\del x^i} \in \Omega_x.
\end{equation}
Observe from \eqref{eq:g_v} and \eqref{eq:Leg} that
$g_{\nabla(-f)}(\nabla(-f)(x),v) =-\dd f(v)$ holds for any $v \in T_xM$.
For a $C^2$-temporal function $f\colon M \lra \R$ and $x \in M$, its \emph{Hessian} $\nabla^2 (-f)\colon T_xM \lra T_xM$ is defined by
\begin{equation*}
\nabla^2 (-f)(v) :=D^{\nabla(-f)}_v \bigl[ \nabla(-f) \bigr].
\end{equation*}
This spacetime Hessian has the following symmetry (see, e.g., \cite[Lemma~4.12]{LMO2}):
\begin{equation}\label{eq:symm}
g_{\nabla(-f)}\bigl( \nabla^2 (-f)(v),w \bigr) =g_{\nabla(-f)} \bigl( v,\nabla^2 (-f)(w) \bigr)
\end{equation}
for all $v,w \in T_xM$.
Then we define the \emph{d'Alembertian} (or the \emph{spacetime Laplacian})
as the trace of the Hessian:
\begin{equation*}
\square(-f) :=\trace \bigl[ \nabla^2 (-f) \bigr].
\end{equation*}
We remark that the d'Alembertian $\square$ is not elliptic but hyperbolic,
and is nonlinear in our Finsler setting (since the Legendre transform is nonlinear).

\begin{remark}\label{rm:Hess}
For $v \in T_xM$ and the geodesic $\eta\colon (-\ve,\ve) \lra M$ with $\dot{\eta}(0)=v$,
the second order derivative $(-f \circ \eta)''(0)$ does not coincide with $g_{\nabla(-f)}(\nabla^2 (-f)(v),v)$ in general.
They coincide in Berwald spacetimes thanks to the fiberwise constancy of the connection coefficients $\Gamma^i_{jk}$
(see \cite[\S 12.1]{Obook} for the positive definite case).
\end{remark}

\subsection{Weighted Finsler spacetimes}\label{ssc:wLF}

There are two ways of putting a \emph{weight} on a Finsler spacetime $(M,L)$.
The first one, reasonable from the point of view of geometric analysis, is to choose a positive $C^\infty$-measure $\fm$ on $M$.
The other possibility is to employ a $C^{\infty}$-function $\Psi \colon M \lra \R$.
They are mutually equivalent in the Lorentzian case via the relation $\fm=\e^{-\Psi} \vol_g$, where $\vol_g$ is the volume measure induced from the Lorentzian metric $g$.
In the Lorentz--Finsler setting, however, we do not have a canonical measure like $\vol_g$ (see \cite{ORand,Obook} for the positive definite case).
In this article, we consider the first way because we will make use of the linear divergence operator induced by a measure.

\begin{remark}\label{rm:meas}
One way to unify the above two cases is to consider a general $0$-homogeneous function
$\psi\colon \ol{\Omega} \setminus \{0\} \lra \R$ as in  \cite{LMO1,LMO2}.
However, even a weighted Laplacian on a Riemannian manifold corresponding to such a general function $\psi$ is yet to be developed (see the end of \cite[\S 3.3]{LMO2}).
\end{remark}

Define a function $\psi_{\fm} \colon \ol{\Omega} \setminus \{0\} \lra \R$ associated with a positive $C^\infty$-measure $\fm$ on $M$ (i.e., its density function is positive and $C^\infty$ in each local chart) by
\[ \fm(\dd x) =\e^{-\psi_{\fm} \circ \dot{\eta}} \sqrt{-\det\bigl[ \bigl( g_{ij}(\dot{\eta}) \bigr) \bigr]}
 \,\dd x^1 \cdots \dd x^n \]
along causal geodesics $\eta$.
In other words, given a causal vector field $V$ on an open set $U \subset M$
such that every integral curve of $V$ is a geodesic,
we have $\fm =\e^{-\psi_\fm \circ V} \vol_{g_V}$ on $U$.

\begin{definition}[Weighted Ricci curvature]\label{df:wRic}
Given $v \in \overline{\Omega} \setminus \{0\}$,
let $\eta\colon (-\ve,\ve) \lra M$ be the causal geodesic with $\dot{\eta}(0)=v$.
Then we define the \emph{weighted Ricci curvature} by
\begin{equation*}
\Ric_N(v) :=\Ric(v) +(\psi_\fm \circ \dot{\eta})''(0) -\frac{(\psi_\fm \circ \dot{\eta})'(0)^2}{N-n}
\end{equation*}
for $N \in \R \setminus \{n\}$.
We also define $\Ric_{\infty}(v) :=\Ric(v) +(\psi_\fm \circ \dot{\eta})''(0)$,
$\Ric_n(v) :=\lim_{N \downarrow n} \Ric_N(v)$, and $\Ric_N(0):=0$.
\end{definition}

We remark that, in the positive definite case, $(\psi_\fm \circ \dot{\eta})'(0)$ is called the \emph{$\mathbf{S}$-curvature} at $v$ associated with $\fm$ (see \cite{Obook,Shlec}).

We will say that $\Ric_N \ge K$ holds \emph{in timelike directions} for some $K \in \R$
if we have $\Ric_N(v) \ge KF^2(v) =-2KL(v)$ for all $v \in \Omega$ (recall \eqref{eq:LtoF} for $F$).
By definition, we have the monotonicity
\[ \Ric_n(v) \le \Ric_N(v) \le \Ric_{\infty}(v) \le \Ric_{N'}(v) \]
for $N \in (n,\infty)$ and $N' \in (-\infty,n)$.
The following notion was introduced in \cite{LMO1}.

\begin{definition}[$\ez$-range]\label{df:eran}
Given $N \in (-\infty,1] \cup [n,\infty]$, we consider $\ez \in \R$ in the following
\emph{$\ez$-range}:
\begin{equation}\label{eq:erange}
\epsilon=0 \,\text{ for } N=1, \qquad
 |\epsilon| < \sqrt{\frac{N-1}{N-n}} \,\text{ for } N \neq 1,n, \qquad
 \ez \in \R \,\text{ for } N=n.
\end{equation}
The associated constant $c =c(N,\ez)$ is defined as
\begin{equation}\label{eq:c}
c(N,\ez):= \frac{1}{n}\bigg( 1-\ez^2\frac{N-n}{N-1} \bigg) >0 \,\text{ for } N \neq 1, \qquad
 c(1,0):= \frac{1}{n}.
\end{equation}
\end{definition}

Note that $\ez=1$ is admissible only for $N \in [n,\infty)$, while $\ez=0$ is always admissible.
By the $\ez$-range, we can unify results for constant and variable curvature bounds
into a single framework.
See \cite{LMO1} for singularity theorems and \cite{LMO2} for comparison theorems.

A d'Alembertian comparison theorem with $\ez$-range was established in \cite{LMO2} as follows.
For a temporal function $f\colon M \lra \R$ and $x \in M$, we define the \emph{weighted d'Alembertian} associated with $\fm$ by
\begin{equation}\label{eq:wdA}
\square_\fm (-f)(x) :=\square(-f)(x) -(\psi_\fm \circ \dot{\eta})'(0),
\end{equation}
where $\eta\colon (-\ve,\ve) \lra M$ is the timelike geodesic with $\dot{\eta}(0)=\nabla(-f)(x)$.
One can alternatively define $\square_\fm(-f)$ as $\div_\fm(\nabla(-f))$; the \emph{divergence} $\div_\fm$ with respect to the measure $\fm$ (independent of $L$) is given in coordinates by
\begin{equation}\label{eq:div}
\div_\fm V :=\sum_{i=1}^n \bigg\{ \frac{\del V^i}{\del x^i} +V^i \frac{\del\Phi}{\del x^i} \bigg\}
\end{equation}
for differentiable vector fields $V$, where $\fm=\e^\Phi \,\dd x^1 \cdots \dd x^n$.
For test functions $\phi \in C^\infty_c(M)$ of compact support, we have  the integration by parts formula:
\[ \int_M \phi \cdot \square_\fm (-f) \,\dd\fm
 =-\int_M \dd\phi \bigl( \nabla(-f) \bigr) \,\dd\fm. \]
The equivalence of these two definitions of $\square_\fm(-f)$ can be seen, for example, in the same way as the positive definite case (see \cite[Lemma~12.4]{Obook}).

The \emph{d'Alembertian} (\emph{Laplacian}) \emph{comparison theorem} (\cite[Theorem~5.8]{LMO2}, \cite[Theorem~3.1]{LMO3}) for nonnegatively curved spacetimes asserts the following.

\begin{theorem}[d'Alembertian comparison theorem]\label{th:Lcomp}
Let $(M,L,\fm)$ be a weighted Finsler spacetime of dimension $n \ge 2$,
$N \in (-\infty,1] \cup [n,\infty]$, and $\ez \in \R$ belong to the $\ez$-range \eqref{eq:erange}.
Suppose that $\Ric_N \ge 0$ holds in timelike directions.
Then, for any $z \in M$, the time separation function $f(x):=\tau(z,x)$ satisfies
\begin{equation}\label{eq:Lcomp}
\square_\fm (-f) \big( \eta(t) \big)
 \le \e^{\frac{2(\ez -1)}{n-1} \psi_\fm(\dot{\eta}(t))}
 \bigg( c\int_0^t \e^{\frac{2(\ez -1)}{n-1} \psi_\fm(\dot{\eta}(s))} \,\dd s \bigg)^{-1}
\end{equation}
for any maximizing timelike geodesic $\eta\colon [0,T) \lra M$ of unit speed
$($i.e., $F(\dot{\eta})\equiv 1)$ emanating from $z$
and any $t \in (0,T)$, where $c=c(N,\ez)$ as in \eqref{eq:c}.
\end{theorem}

In fact, in \cite{LMO2}, a general $0$-homogeneous function was considered as a weight (recall Remark~\ref{rm:meas}).
(Recall also that $\dim M=n+1$ in \cite{LMO2,LMO3}.)
We refer to  \cite{Ca} and  \cite{WW} for the weighted Lorentzian case.

We will also need the reverse version of Theorem~\ref{th:Lcomp}.
The \emph{reverse Lorentz--Finsler structure} of $L$ is defined as $\rev{L}(v):=L(-v)$,
and we put an arrow $\leftarrow$ on a quantity associated with $\rev{L}$.
The Lorentz--Finsler manifold $(M,\rev{L})$ is time oriented by $-X$,
so that $\rev{\Omega}=-\Omega$.
The corresponding weighted d'Alembertian satisfies $\rev{\square}_\fm f=-\square_\fm(-f)$
for temporal functions $f$
(since $\rev{\psi}_{\fm}(v)=\psi_{\fm}(-v)$).
We also remark that $\rev{\Ric}_N(v)=\Ric_N(-v)$ for $v \in \rev{\Omega}$, and hence
the timelike curvature bound $\Ric_N \ge 0$ is equivalent to $\rev{\Ric}_N \ge 0$.

\begin{corollary}[Reverse version]\label{cr:Lcomp}
Let $(M,L,\fm)$, $N$ and $\epsilon$ be as in Theorem~$\ref{th:Lcomp}$.
Then, for any $z \in M$, the function $\ol{f}(x):=\tau(x,z)$ satisfies
\begin{equation}\label{eq:Lcomp'}
\rev{\square}_\fm (-\ol{f}) \big( \eta(-t) \big)
 =-\square_\fm \ol{f} \big( \eta(-t) \big)
 \le \e^{\frac{2(\ez -1)}{n-1} \psi_\fm(\dot{\eta}(-t))}
 \bigg( c\int_{-t}^0 \e^{\frac{2(\ez -1)}{n-1} \psi_\fm(\dot{\eta}(s))} \,\dd s \bigg)^{-1}
\end{equation}
for any maximizing timelike geodesic $\eta\colon (-T,0] \lra M$ of unit speed with $\eta(0)=z$
and any $t \in (0,T)$.
\end{corollary}

Note that the reverse curve $\bar{\eta}(t):=\eta(-t)$ is a maximizing timelike geodesic of unit speed with respect to $\rev{L}$.
In the application of the d'Alembertian comparison theorem to splitting theorems,
we need to require that the right-hand side of \eqref{eq:Lcomp} (and \eqref{eq:Lcomp'}) tends to $0$.
To be precise, we will assume the following completeness condition introduced in \cite{LMO1}, adapted to our usage in splitting theorems.

\begin{definition}[$\ez$-completeness]\label{df:cmplt}
A weighted Finsler spacetime $(M,L,\fm)$ is said to be \emph{future timelike $\ez$-complete} if, given any future complete maximizing timelike geodesic $\eta\colon [0,\infty) \lra M$, there is a neighborhood $U$ of $\eta(0)$ such that, for any sequence $t_k \to \infty$ admitting maximizing geodesics $\zeta_k\colon [0,\tau(x,\eta(t_k))] \lra M$ from $x \in U$ to $\eta(t_k)$ of unit speed, we have
\[ \lim_{k \to \infty} \int_0^{\tau(x,\eta(t_k))} \e^{\frac{2(\ez -1)}{n-1} \psi_\fm(\dot{\zeta}_k(s))} \,\dd s =\infty. \]
The \emph{past timelike $\ez$-completeness} will mean the future timelike $\ez$-completeness of $(M,\rev{L},\fm)$.
We say that $(M,L,\fm)$ is \emph{timelike $\ez$-complete} if it is both future and past timelike $\ez$-complete.
\end{definition}

Note that $\tau(x,\eta(t_k)) \to \infty$ by the reverse triangle inequality, and thus the timelike $1$-completeness always holds.
Timelike $0$-completeness recovers the $\psi_\fm$-completeness introduced in  \cite{Wy}
in the Riemannian case (see  \cite{WW} for the Lorentzian case).
A typical situation is that $\ez<1$ and $\psi_\fm$ is bounded above (see  \cite{Ca}, where $N=\infty$).

In the timelike geodesically complete case (see Definition~\ref{df:gcmplt}), for our purpose, it is in fact sufficient to assume
\begin{equation}\label{eq:cmplt}
\int_0^\infty \e^{\frac{2(\ez -1)}{n-1} \psi_\fm(\dot{\eta}(s))} \,\dd s =\infty
\end{equation}
for every future complete timelike geodesic $\eta\colon [0,\infty) \lra M$.

\section{The $p$-d'Alembert operator}\label{sc:p-dA}

This section is devoted to a key new ingredient, the $p$-d'Alembertian, which is a modification of the d'Alembertian \eqref{eq:wdA} in a similar manner to the $p$-Laplacian in the positive definite case (for $p>1$), but with $p<1$.
We refer to \cite{octet,Br,quintet} and \cite{McCann:2020,MS:22} for the (possibly singular) Lorentzian setting.

We first introduce the \emph{$q$-Lagrangian} $L_q$ as in  \cite{McCann:2020} and  \cite{BO} with $q \in (0,1)$ (see also  \cite{Suhr} for the case $q=1$ and  \cite[\S 3.1]{Min-causality} for a more general study including the Finsler case).
Define $L_q\colon TM \lra (-\infty,0] \cup \{\infty\}$ as
\begin{equation*}
L_q(v) := \begin{cases}
 -\dfrac{1}{q} \bigl( -2L(v) \bigr)^{q/2} & \text{if } v \in \overline{\Omega}, \\
 \infty & \text{otherwise}.
\end{cases}
\end{equation*}
Observe that $L_q(v) =-F(v)^q/q$ for $v \in \overline{\Omega}$.
We remark that  $L$  is not convex in the radial direction,
while  for $L_q$ the following holds (see, e.g., \cite[Lemma~3.1]{BO}):
\begin{lemma}\label{lm:Lag}
For any $q \in (0,1)$ and $x \in M$, the Lagrangian $L_q$ is convex on $T_xM$ and strictly convex on $\Omega_x$.
\end{lemma}

The convex dual $H_p \colon T^*M \lra [0,\infty]$ of $L_q$,
called the \emph{$p$-Hamiltonian}, is given by
\begin{equation*}
H_p(\omega) := \sup_{v \in TM} \bigl( \omega(v) -L_q(v) \bigr),
\end{equation*}
where $q^{-1}+p^{-1}=1$ and $\omega(v)$ is the canonical pairing.
Note that $p<0$ and we have
\begin{equation}\label{eq:H_p}
H_p(\omega) = \begin{cases}
 -\dfrac{1}{p} \bigl( -2L^*(\omega) \bigr)^{p/2} & \text{if } \omega \in \Omega^*, \\
 \infty & \text{otherwise}.
\end{cases}
\end{equation}

Here and henceforth, we fix a choice of $q \in (0,1)$ and corresponding dual $p < 0$ as above.

\begin{remark}[Ranges of $q,p$]\label{rm:pq}
One can equally well consider $q<0$ and thus $p \in (0,1)$.
The convexity properties of the Lagrangian and the corresponding Hamiltonian remain unchanged (cf.  \cite[Example~3.1]{Min-causality}); moreover, the ellipticity of the $p$-d'Alembertian is also preserved, see \cite{octet, quintet} for details. Note that if $q \in (0,1)$ and thus $p < 0$, $L_q(v) = 0$ for $v \in \partial \Omega$, while $H_p(\omega) = \infty$ for $\omega \in \partial \Omega^*$. These behaviors at the boundary become interchanged if one instead considers $q < 0$ and thus $p \in (0,1)$. For the purposes of our proof of the splitting theorem, we restrict ourselves to a fixed choice of $q \in (0,1)$ and $p < 0$, and will utilize the ellipticity of the $p$-d'Alembertian for this $p$ to obtain our desired results for the Busemann functions. There is no advantage in varying $q$ and $p$ throughout the proof.
\end{remark}

Define the \emph{$p$-energy functional} for a temporal function $f$,
associated to the measure $\fm$, by
\begin{equation*}
\cE_p(-f) := \int_M H_p(-\dd f) \,\dd \fm
 = -\frac{1}{p} \int_M \bigl(  -2L^*(-\dd f) \bigr)^{p/2} \,\dd\fm
 =-\frac{1}{p} \int_M F^*(-\dd f)^p \,\dd\fm.
\end{equation*}
Then, $\cE_p$ is convex in the cone consisting of temporal functions.
Thus, the corresponding \emph{$p$-d'Alembertian} (or the \emph{spacetime $p$-Laplacian}) defined by
\begin{equation}\label{eq:p-dA}
\begin{split}
\square_{\fm,p}(-f)
&:=\div_{\fm} \bigl( F^*(-\dd f)^{p-2} \cdot \nabla(-f) \bigr) \\
&= F^*(-\dd f)^{p-2} \cdot \square_\fm(-f) +\dd\bigl[ F^*(-\dd f)^{p-2} \bigr] \bigl( \nabla(-f) \bigr)
\end{split}
\end{equation}
is elliptic on smooth temporal functions. 
In fact,  for any smooth function $\phi \in C^{\infty}_c(M)$ of compact support, $f+t\phi$ is temporal provided that $|t|$ is sufficiently small. Since the integrand is smooth in $t$ and uniformly bounded on the compact support of $\phi$ for small $|t|$, by Lebesgue's dominated convergence theorem and \eqref{eq:Leg} we deduce that
\begin{align*}
\lim_{t \to 0} \frac{\cE_p(-f-t\phi) -\cE_p(-f)}{t}
&= -\frac{1}{2} \int_M \bigl( -2L^*(-\dd f) \bigr)^{(p-2)/2} \cdot 2\dd\phi \bigl( \nabla(-f) \bigr) \,\dd\fm \\
&= -\int_M \dd\phi \bigl( F^*(-\dd f)^{p-2} \cdot \nabla(-f) \bigr) \,\dd\fm \\
&= \int_M \phi \cdot \div_\fm \bigl( F^*(-\dd f)^{p-2} \cdot \nabla(-f) \bigr) \,\dd\fm.
\end{align*}
Hence, 
\[ \cE_p(-f-t\phi) =\cE_p(-f)
 -t \int_M (-\phi) \cdot \square_{\fm,p}(-f) \,\dd\fm +o(t), \]
which identifies $\square_{\fm,p}$ as the Euler--Lagrange operator associated with $\cE_p$.

\begin{remark}\label{rm:ellip}
In the right-hand side of \eqref{eq:p-dA}, we observe from \eqref{eq:g_eta} or \eqref{eq:g_X} that
\begin{align*}
\dd\bigl[ F^*(-\dd f)^{p-2} \bigr] \bigl( \nabla(-f) \bigr)
&= \dd\Bigl[ \Bigl( -g_{\nabla(-f)}\bigl( \nabla(-f),\nabla(-f) \bigr) \Bigr)^{(p-2)/2} \Bigr] \bigl( \nabla(-f) \bigr) \\
&= -(p-2)F^*(-\dd f)^{p-4} g_{\nabla(-f)}\Bigl( D^{\nabla(-f)}_{\nabla(-f)}\bigl[ \nabla(-f) \bigr],\nabla(-f) \Bigr) \\
&= -(p-2)F^*(-\dd f)^{p-4} g_{\nabla(-f)} \Bigl( \nabla^2(-f)\bigl( \nabla(-f) \bigr),\nabla(-f) \Bigr)
\end{align*}
(a similar calculation in the positive definite case can be found in \cite[\S 3.1]{OSbw}).
Hence,
\begin{equation}\label{eq:p-dA'}
\frac{\square_{\fm,p}(-f)}{F^*(-\dd f)^{p-2}}
=\square_{\fm}(-f) -(p-2)g_{\nabla(-f)} \biggl( \nabla^2(-f)\biggl( \frac{\nabla(-f)}{F^*(-\dd f)} \biggr),\frac{\nabla(-f)}{F^*(-\dd f)} \biggr).
\end{equation}
In a suitable local chart $(x^i)_{i=1}^n$ around $x \in M$ such that $(\del/\del x^i)_{i=1}^n$ is orthonormal with respect to $g_{\nabla(-f)}$ with $\del/\del x^1 =\nabla(-f)/F^*(-\dd f)$ at $x$,
recalling the coordinate expression of $\nabla(-f)$ in \eqref{eq:grad}, we can write down the second order part of \eqref{eq:p-dA'} at $x$ as
\[ -\frac{\del^2 (-f)}{(\del x^1)^2} +\sum_{i=2}^n \frac{\del^2 (-f)}{(\del x^i)^2} -(p-2)\frac{\del^2 (-f)}{(\del x^1)^2}
=(1-p)\frac{\del^2 (-f)}{(\del x^1)^2} +\sum_{i=2}^n \frac{\del^2 (-f)}{(\del x^i)^2}. \]
Therefore, $\square_{\fm,p}$ is elliptic exactly when $p<1$.
We also remark that, in the positive definite case, the $p$-Laplacian $\Delta_p f:=\div_\fm (F^*(\dd f)^{p-2} \cdot \nabla f)$ is usually considered for $p \in (1,\infty)$.
\end{remark}

For $f(x):=\tau(z,x)$ and $\ol{f}(x):=\tau(x,z)$, if we replace $\square_\fm$ with $\square_{\fm,p}$
in Theorem~\ref{th:Lcomp} as well as Corollary~\ref{cr:Lcomp},
since $F^*(-\dd f) \equiv 1$ (resp.\ $F^*(\dd\ol{f}) \equiv 1$) on $I^+(z)$ (resp.\ $I^-(z)$) outside the future (resp.\ past) timelike cut locus, we immediately obtain the following.

\begin{lemma}[$p$-d'Alembertian comparison estimates]\label{lm:p-comp}
Let $(M,L,\fm)$, $N$ and $\ez$ be as in Theorem~$\ref{th:Lcomp}$.
Then, for any $z \in M$, the function $f(x):=\tau(z,x)$ satisfies
\begin{equation}\label{eq:p-comp}
\square_{\fm,p} (-f)\big( \eta(t) \big)
 \le \e^{\frac{2(\ez -1)}{n-1} \psi_\fm(\dot{\eta}(t))}
 \bigg( c\int_0^t \e^{\frac{2(\ez -1)}{n-1} \psi_\fm(\dot{\eta}(s))} \,\dd s \bigg)^{-1}
\end{equation}
for any maximizing timelike geodesic $\eta\colon [0,T) \lra M$ of unit speed
with $\eta(0)=z$ and any $t \in (0,T)$.
Moreover, the function $\ol{f}(x):=\tau(x,z)$ satisfies
\begin{equation}\label{eq:p-comp'}
\rev{\square}_{\fm,p} (-\ol{f}) \big( \eta(-t) \big)
 =-\square_{\fm,p} \ol{f} \big( \eta(-t) \big)
 \le \e^{\frac{2(\ez -1)}{n-1} \psi_\fm(\dot{\eta}(-t))}
 \bigg( c\int_{-t}^0 \e^{\frac{2(\ez -1)}{n-1} \psi_\fm(\dot{\eta}(s))} \,\dd s \bigg)^{-1}
\end{equation}
for any maximizing timelike geodesic $\eta\colon (-T,0] \lra M$ of unit speed with $\eta(0)=z$
and any $t \in (0,T)$.
\end{lemma}

\begin{remark}[Unweighted setting]\label{rm:unwght}
For completeness, we also discuss the $p$-d'Alembertian in the unweighted situation.
Let us set
\[ \nabla_p(-f) :=\sum_{i=1}^n \frac{\del H_p}{\del \omega_i}(-\dd f)\frac{\del}{\del x^i}
=F^*(-\dd f)^{p-2} \cdot \nabla(-f). \]
Then, it is natural to define
\begin{equation*}
\begin{split}
\square_p(-f) &:=\trace \Bigl[ D^{\nabla(-f)} \bigl[ \nabla_p (-f) \bigr] \Bigr] \\
&= F^*(-\dd f)^{p-2} \cdot \square(-f) +\dd\bigl[ F^*(-\dd f)^{p-2} \bigr] \bigl( \nabla(-f) \bigr).
\end{split}
\end{equation*}
Note that $\nabla_2(-f)=\nabla(-f)$ and $\square_2(-f)=\square(-f)$.
\end{remark}

\section{Busemann functions}\label{sc:busemann}

In this section, we summarize some properties of the Busemann function associated with a timelike straight line established in  \cite{LMO3}, under the assumption of timelike geodesic completeness, along the lines of  \cite{GH} (see also \cite[Chapter~14]{BEE:1996}).
The globally hyperbolic case will be discussed later in Subsection~\ref{ssc:ghyp}. 

\subsection{Rays and generalized co-rays}\label{ssc:ray}

\begin{definition}[Timelike geodesic completeness]\label{df:gcmplt}
A Finsler spacetime $(M,L)$ is said to be \emph{future timelike geodesically complete}
if any timelike geodesic $\eta\colon [0,1] \lra M$ can be extended to a geodesic
$\tilde{\eta}\colon [0,\infty) \lra M$.
We say that $(M,L)$ is \emph{timelike geodesically complete}
if both $(M,L)$ and $(M,\rev{L})$ are future timelike geodesically complete
(in other words, the above $\eta$ is extended to a geodesic $\tilde{\eta}\colon \R \lra M$).
\end{definition}

A future inextendible causal geodesic $\eta\colon [0,T) \lra M$ (with $0 < T \leq \infty$) is called a \emph{ray}
if each of its segments is maximizing, i.e.,
$\bL(\eta|_{[a,b]})=\tau(\eta(a),\eta(b))$ for all $0\le a\le b < T$.
When $\eta$ is future complete (i.e., $T=\infty$, which is always the case if $(M,L)$ is future timelike geodesically complete), we define
\begin{equation}\label{eq:I(eta)}
I^-(\eta) := \bigcup_{t>0} I^-\big( \eta(t) \big), \qquad
I(\eta) :=I^+\big( \eta(0) \big) \cap I^-(\eta).
\end{equation}
Then $I(\eta)$ is an open set including $\eta((0,\infty))$, and $\tau(x,y)<\infty$ for all $x,y \in I(\eta)$ with $x \le y$ (this follows from the reverse triangle inequality as well as the fact that $\tau(\eta(a),\eta(b)) < \infty$ for all $0 \le a \le b <\infty$).

In the rest of this section, we suppose that $(M,L)$ is future timelike geodesically complete.
Let $\eta\colon [0,\infty) \lra M$ be a timelike ray.
Rays asymptotic to a given ray, called co-rays, play an important role in splitting theorems.
In the non-globally hyperbolic case, due to the possible absence of maximizing curves, we need to consider also generalized co-rays (see \cite{BeemEhrlichMarkvorsenGalloway84,BEMG85,EG}).

We recall that a sequence $\zeta_k \colon [0,a_k] \lra M$ of causal curves is said to be {\em limit maximizing} if
$\bL(\zeta_k)\geq \tau(\zeta_k(0), \zeta_k(a_k)) - \ve_k$ for some positive sequence $\ve_k \to 0$ (see, e.g., \cite[p.~210]{EG}).
Then, as a consequence of a limit curve argument, if a limit maximizing sequence $\zeta_k\colon [0,a_k] \lra M$ of causal curves satisfies $\zeta_k(0)\to x$ and $\tau(\zeta_k(0),\zeta_k(a_k)) \to \infty$, then there exists a subsequence of $(\zeta_k)$ that, up to  reparametrizations by arclength with respect to an auxiliary complete Riemannian metric on $M$, converges uniformly on compact subintervals of $[0,\infty)$ to a ray $\zeta\colon [0,\infty) \lra M$ emanating from $x$ (see  \cite[Lemma~4.4]{LMO3}, \cite[Lemma~2.4]{GH}).
\begin{definition}[Generalized co-rays]\label{df:coray}
Let $x \in I(\eta)$.
If a limit maximizing sequence of causal curves $\zeta_k\colon [0,a_k] \lra M$ satisfies
$\zeta_k(0) \to x$ and $\zeta_k(a_k)=\eta(t_k)$ for some $t_k>0$ with $t_k \to \infty$,
then its limit curve $\zeta\colon [0,\infty) \lra M$ is called a \emph{generalized co-ray} of $\eta$.
If each $\zeta_k$ is maximizing, then we call $\zeta$ a \emph{co-ray} of $\eta$.
A co-ray $\zeta$ with $\zeta_k(0)=x$ for all $k$ is called an \emph{asymptote} of $\eta$.
\end{definition}

The limit curve $\zeta$ is indeed a ray as we mentioned above.
To avoid convergence to a lightlike ray, we introduce the following condition.

\begin{definition}[Generalized timelike co-ray condition (GTCC)]\label{df:GTCC}
Let $x \in I(\eta)$.
We say that the \emph{generalized timelike co-ray condition}
(\emph{GTCC} for short) for $\eta$ holds at $x$
if any generalized co-ray of $\eta$ emanating from $x$ is timelike.
\end{definition}

The GTCC implies the following fine properties (see \cite[Lemmas~4.8, 4.9]{LMO3}, \cite[Lemmas~3.3, 3.4]{GH}).

\begin{lemma}[Existence of maximizing geodesics]\label{lm:GH3.3}
Suppose that the GTCC for $\eta$ holds at some $x \in I(\eta)$.
Then, there exist a neighborhood $U$ of $x$ and $T>0$ such that, for any $z \in U$ and $t>T$, there exists a maximizing timelike geodesic from $z$ to $\eta(t)$.
Moreover, every $z \in U$ admits a timelike asymptote of $\eta$ emanating from $z$.
\end{lemma}

The GTCC holds on a neighborhood of $\eta$ as follows (see \cite[Proposition~4.10]{LMO3}, \cite[Proposition~5.1, Corollary~5.2]{GH}).

\begin{proposition}[GTCC near rays]\label{pr:GH5.1}
Any generalized co-ray emanating from $\eta(a)$ with $a>0$
necessarily coincides with $\eta$.
In particular, the GTCC holds on an open set including $\eta((0,\infty))$.
\end{proposition}

\subsection{Busemann functions for rays}\label{ssc:rays}

We continue to assume future timelike geodesic completeness,
and let $\eta\colon [0,\infty)\lra M$ be a timelike ray of unit speed ($F(\dot{\eta}) \equiv 1$).
We introduce a central ingredient of the proof of the splitting theorem.

\begin{definition}[Busemann functions]\label{df:Buse}
Define the \emph{Busemann function} $\bb_{\eta}\colon M \lra [-\infty,\infty]$
associated with $\eta$ by
\[ \bb_{\eta}(x) :=\lim_{t \to \infty} \bb_{\eta,t}(x), \qquad
 \text{where}\,\ \bb_{\eta,t}(x) :=t-\tau\bigl( x,\eta(t) \bigr). \]
\end{definition}

The limit above always exists in $\R \cup \{\pm \infty\}$.
Precisely, if $x \not\in I^-(\eta)$, then $\tau(x,\eta(t))=0$ for all $t$ and hence $\bb_{\eta}(x)=\infty$.
If $x \in I^-(\eta)$, then the reverse triangle inequality implies that, for large $s<t$,
\[ \bb_{\eta,t}(x) \le t-\big\{ \tau\big( x,\eta(s) \big) +\tau\big( \eta(s),\eta(t) \big) \big\}
 =\bb_{\eta,s}(x). \]
Hence, $\bb_{\eta,t}(x)$ is non-increasing in $t$ and converges to $\bb_{\eta}(x) \in \R \cup \{-\infty\}$.
One can also see $\bb_\eta(x) \in \R$ for $x \in I(\eta)$ since
\[ \bb_{\eta,t}(x) \ge t +\tau\bigl( \eta(0),x \bigr) -\tau\bigl( \eta(0),\eta(t) \bigr) =\tau\bigl( \eta(0),x \bigr). \]

It follows from the reverse triangle inequality that $\bb_\eta(y) \ge \bb_\eta(x) +\tau(x,y)$ for $x \le y$.
Moreover, since $\tau$ is lower semi-continuous, $\bb_{\eta}$ is upper semi-continuous in $I^-(\eta)$.
We summarize some further fundamental properties of $\bb_\eta$ along timelike asymptotes.

\begin{lemma}[Properties of $\bb_\eta$ along asymptotes]\label{lm:zeta}
Let $\zeta\colon [0,\infty) \lra M$ be a timelike asymptote of $\eta$ of unit speed
with $x:=\zeta(0) \in I(\eta)$.
\begin{enumerate}[{\rm (i)}]
\item\label{it:zeta1}
For each $t>0$,
\[ \rho(z) :=\bb_\eta(x) +\bb_{\zeta,t}(z) =\bb_\eta(x) +t -\tau\big( z,\zeta(t) \big) \]
is an \emph{upper support function} for $\bb_\eta$ at $x$ in the sense that
$\rho \ge \bb_\eta$ on $I^+(\eta(0)) \cap I^-(\zeta(t))$ and $\rho(x)=\bb_\eta(x)$.

\item\label{it:zeta2}
We have $\bb_\eta(\zeta(t)) =\bb_\eta(x)+t$ for all $t \ge 0$.
In particular, $I(\zeta) \subset I(\eta)$ holds.

\item\label{it:zeta3}
If $\bb_\eta$ is differentiable at $x$, then we have $\zeta(t)=\exp_x(t\nabla(-\bb_\eta)(x))$, and $\zeta$ is the unique timelike asymptote of $\eta$ of unit speed emanating from $x$.
In particular, $F(\nabla(-\bb_\eta)(x))=1$ holds.

\item\label{it:zeta4}
For each $t>0$, $\bb_\eta$ is differentiable at $\zeta(t)$ with $\nabla(-\bb_\eta)(\zeta(t)) =\dot{\zeta}(t)$.
\end{enumerate}
\end{lemma}

\begin{proof}
We refer to \cite[Lemma~5.2]{LMO3} (and  \cite[\S 2.2]{Eschenburg:1988},  \cite[Lemma~2.5, Proposition~2.6]{GH}) for \eqref{it:zeta1} and \eqref{it:zeta2}.
Then, \eqref{it:zeta3} is a direct consequence of \eqref{it:zeta2}.
Indeed, if $\dd\bb_\eta(v)=\dd\bb_\eta(w)=1$ for distinct timelike vectors $v,w \in \Omega_x \cap F^{-1}(1)$, then we have $\dd\bb_\eta((v+w)/2)=1$ and $F((v+w)/2)>1$.
However, this is a contradiction since $\bb_\eta(y) \ge \bb_\eta(x) +\tau(x,y)$ for $x \le y$.
Finally, \eqref{it:zeta4} is seen as in \cite[\S 2.2]{Eschenburg:1988}, by noticing that $\bb_\eta(\zeta(T))-\tau(\cdot,\zeta(T))$ and $\bb_\eta(x)+\tau(x,\cdot)$ are upper and lower support functions for $\bb_\eta$ at $\zeta(t)$ for $T>t$, respectively.
\end{proof}

Under the GTCC, $\bb_\eta$ is in fact differentiable almost everywhere near $\eta$ (see \cite[Theorem~5.3]{LMO3}, \cite[Theorem~3.7]{GH}).

\begin{theorem}[Regularity of $\bb_{\eta}$]\label{th:GH3.7}
Assume that the GTCC for $\eta$ holds at $x \in I(\eta)$.
Then, $\bb_{\eta}$ is Lipschitz continuous, and hence differentiable almost everywhere, on a neighborhood of $x$.
\end{theorem}

We remark that the Lipschitz continuity is understood with respect to an auxiliary Riemannian metric.

Recall that a function $u$ defined on an open set $U \subset M$ is said to be \emph{semiconvex} with respect to a (background) Riemannian metric $h$ if there exists a constant $c \in \mathbb{R}$ such that for every $x \in U$,
\begin{equation*}
    \liminf_{w \to 0} \frac{u(\exp_x^h(w)) + u(\exp_x^h(-w)) - 2u(x)}{h(w,w)} \geq c,
\end{equation*}
where $w \in T_xM$ and $\exp^h$ denotes the exponential map for $h$.
We say that $u$ is \emph{semiconcave} if $-u$ is semiconvex.
With the help of the semiconvexity of the time separation function in  \cite[Proposition~3.9]{BO} (along the lines  of  \cite[Proposition~3.4]{McCann:2020}), we have a uniform estimate on the semiconcavity of $\bb_{\eta,t}$ (cf. \cite[Proposition~5]{quintet}) as follows. 

\begin{lemma}[Equi-semiconcavity of $\bb_{\eta,t}$]\label{lm:concave}
Fix a smooth Riemannian metric $h$ of $M$, and let $W \subset I(\eta)$ be an open set including $\eta((0,\infty))$ as in Proposition~$\ref{pr:GH5.1}$.
Then, for each $x_0 \in W$, there exist a neighborhood $U$ of $x_0$, $T>0$ and a constant $C=C(x_0,T) \in \R$ such that
\begin{equation*}
\limsup_{r \to 0} \frac{\bb_{\eta,t}(\exp^h_x(-rv)) +\bb_{\eta,t}(\exp^h_x(rv)) -2\bb_{\eta,t}(x)}{r^2 h(v,v)} \le C
\end{equation*}
for all $x \in U$, $v \in T_xM \setminus \{0\}$ and $t \ge 2T$.
\end{lemma}

\begin{proof}
Given $x_0 \in W$, take a neighborhood $U$ of $x_0$ and $T>0$ as in Lemma~\ref{lm:GH3.3}.
We fix $x \in U$ and $t \ge 2T$, and give a modification of Lemma~\ref{lm:zeta}\eqref{it:zeta1} to obtain an upper support function of $\bb_{\eta,t}$.
Let $\zeta\colon [0,\tau(x,\eta(t))] \lra M$ be a maximizing geodesic from $x$ to $\eta(t)$.
Note that
\[ \tau\bigl( x,\eta(t) \bigr)
\ge \tau\bigl( x,\eta(T) \bigr) +\tau\bigl( \eta(T),\eta(t) \bigr)
\ge t-T \ge T \]
by the reverse triangle inequality.
Then, we infer that, for $z$ close to $x$,
\begin{align*}
\bb_{\eta,t}(z) -\bb_{\eta,t}(x)
&= \tau\bigl( x,\eta(t) \bigr) -\tau\bigl( z,\eta(t) \bigr) \\
&\le \tau\bigl( x,\eta(t) \bigr) -\tau\bigl( z,\zeta(T) \bigr) -\tau\bigl( \zeta(T),\eta(t) \bigr) \\
&= \tau\bigl( x,\zeta(T) \bigr) -\tau\bigl( z,\zeta(T) \bigr).
\end{align*}
Thus,
\[ \rho(z):=\bb_{\eta,t}(x) +T -\tau\bigl( z,\zeta(T) \bigr) \]
is an upper support function of $\bb_{\eta,t}$ at $x$.

Hence, for $\gamma(r):=\exp^h_x(rv)$ with $v \in T_xM$ such that $h(v,v)=1$, we have
\begin{align*}
\frac{\bb_{\eta,t}(\gamma(-r)) +\bb_{\eta,t}(\gamma(r)) -2\bb_{\eta,t}(x)}{r^2}
&\le \frac{\rho(\gamma(-r)) +\rho(\gamma(r)) -2\rho(x)}{r^2} \\
&= -\frac{\tau(\gamma(-r),\zeta(T)) +\tau(\gamma(r),\zeta(T)) -2\tau(x,\zeta(T))}{r^2},
\end{align*}
provided that $r>0$ is sufficiently small.
The right-hand side is bounded from above as follows.
Let $P$ be the vector field along the reparametrized geodesic $s \longmapsto \zeta(sT)$, $s \in [0,1]$, such that $P(0)=v$ and $D^{\dot{\zeta}}_{\dot{\zeta}} P \equiv 0$.
We also consider
\[ V(s) :=(1-s) P(s), \qquad
V^{\perp}(s) :=V(s) +g_{\dot{\zeta}}\bigl( V(s),\dot{\zeta}(sT) \bigr) \dot{\zeta}(sT). \]
Observe that $g_{\dot{\zeta}}(V^\perp(s),\dot{\zeta}(sT))=0$ and, by \eqref{eq:g_eta} or \eqref{eq:g_X},
\[ D^{\dot{\zeta}}_{\dot{\zeta}} V^{\perp}(s)
=-P(s) +g_{\dot{\zeta}}\bigl( -P(s),\dot{\zeta}(sT) \bigr) \dot{\zeta}(sT)
=:-P^{\perp}(s). \]
Then, it follows from the calculation in  \cite[Proposition~3.9]{BO} that
\begin{align*}
&\liminf_{r \to 0} \frac{\tau(\gamma(-r),\zeta(T)) +\tau(\gamma(r),\zeta(T)) -2\tau(x,\zeta(T))}{r^2} \\
&\ge -T \int_0^1 \Bigl\{ g_{\dot{\zeta}}(D^{\dot{\zeta}}_{\dot{\zeta}} V^{\perp},D^{\dot{\zeta}}_{\dot{\zeta}} V^{\perp}) -g_{\dot{\zeta}}\bigl( R_{\dot{\zeta}}(V^\perp),V^\perp \bigr) \Bigr\} \,\dd s
+g_{\dot{\zeta}}\bigl( D^{\dot{\zeta}}_{\dot{\gamma}} \dot{\gamma}(0),\dot{\zeta}(0) \bigr) \\
&= -T \int_0^1 \Bigl\{ g_{\dot{\zeta}}(P^\perp,P^\perp) -(1-s)^2 g_{\dot{\zeta}}\bigl( R_{\dot{\zeta}}(P^\perp),P^\perp \bigr) \Bigr\} \,\dd s
+g_{\dot{\zeta}}\bigl( D^{\dot{\zeta}}_{\dot{\gamma}} \dot{\gamma}(0),\dot{\zeta}(0) \bigr).
\end{align*}
Since the right-hand side depends only on smooth quantities on $\zeta([0,T])$, and all such initial vectors $\dot{\zeta}(0)$ belong to a compact set in $\Omega \cap F^{-1}(1)$ (see  \cite[Lemma~3.4]{GH}), it is thus bounded from below by some constant $-C(x_0,T)$.
This yields the desired semiconcavity of $\bb_{\eta,t}$. 
\end{proof}

\subsection{Busemann functions for straight lines}\label{ssc:lines}

In splitting theorems, we assume the existence of a timelike \emph{straight line} $\eta\colon \R \lra M$, i.e., $\tau(\eta(s),\eta(t))=t-s$ holds for all $s<t$.
We will denote by $\bb_{\eta}$ the Busemann function associated with the ray $\eta|_{[0,\infty)}$.
Moreover, the curve $\bar{\eta}(t):=\eta(-t)$, $t \in [0,\infty)$, is a timelike ray of unit speed
with respect to the reverse structure $\rev{L}(v)=L(-v)$ (recall Subsection~\ref{ssc:wLF}).
Hence, we can introduce the corresponding Busemann function as
\[ \ol{\bb}_{\eta}(x):=\lim_{t \to \infty} \big\{ t-\tau\big( \eta(-t),x \big) \big\} \]
($\tau$ is with respect to $L$).
It follows from the reverse triangle inequality that
\begin{equation}\label{eq:b+b}
\bb_{\eta}(x)+\ol{\bb}_{\eta}(x) \ge \lim_{t \to \infty} \big\{ 2t -\tau\big( \eta(-t),\eta(t) \big) \big\} =0,
\end{equation}
and equality holds on $\eta$.

For straight lines we modify $I(\eta)$ in \eqref{eq:I(eta)} into
\[ I(\eta) :=\Bigg( \bigcup_{t>0} I^+\big( \eta(-t) \big) \Bigg)
 \cap \Bigg( \bigcup_{t>0} I^- \big( \eta(t) \big) \Bigg). \]
Note that $I(\eta) =\bigcup_{s<0} I(\eta|_{[s,\infty)})$,
therefore, the previous results considered for rays continue to hold on $I(\eta)$ for timelike straight lines $\eta$.
We infer from Lemma~\ref{lm:zeta}\eqref{it:zeta2} the following (see  \cite[Lemma~5.5]{LMO3}, \cite[Lemma~4.1]{Eschenburg:1988}).

\begin{lemma}\label{lm:line}
Assume that $x \in I(\eta)$ satisfies $\bb_\eta(x)+\ol{\bb}_\eta(x)=0$,
and let $\zeta^+$ and $\zeta^-$ be timelike asymptotes of $\eta$ and $\bar{\eta}$ of unit speed
emanating from $x$, respectively.
Then, the concatenation $\zeta$ of $\zeta^+$ and $\zeta^-$ is a timelike straight line.
\end{lemma}

Precisely, $\zeta^-$ is an asymptote of $\bar{\eta}$ of unit speed with respect to $\rev{L}$,
and the concatenation $\zeta$ is given by $\zeta(t):=\zeta^+(t)$ for $t \ge 0$ and $\zeta(t):=\zeta^-(-t)$ for $t<0$.
We say that such a straight line $\zeta$ is \emph{bi-asymptotic} to $\eta$.

\section{$p$-harmonicity and smoothness of Busemann functions}\label{sc:harm}

With the basic properties of the Busemann function $\bb_\eta$ in the previous section in hand, we proceed to the analysis of $\bb_\eta$ using the $p$-d'Alembertian $\square_{\fm,p}$.
We begin with the wider range $N \in (-\infty,1] \cup [n,\infty]$, and restrict ourselves to $N \in (-\infty,0) \cup [n,\infty]$ at the last step (see Remark~\ref{rm:Wylie}).
\subsection{Maximum principle and $p$-harmonicity}\label{ssc:max}

We first show that, in a rather standard way, $\bb_\eta$ is \emph{superharmonic} with respect to $\rev\square_{\fm,p}$ (cf. \cite[Corollary~8]{quintet}).

\begin{lemma}[$p$-superharmonicity]\label{lm:sharm}
Let $(M,L,\fm)$ be a timelike geodesically complete, weighted Finsler spacetime as in Theorem~$\ref{th:LFsplit}$, but with $N \in (-\infty,1] \cup [n,\infty]$.
Then, the Busemann functions $\bb_\eta,\ol{\bb}_\eta$ satisfy $\rev{\square}_{\fm,p} \bb_\eta \le 0$ and $\square_{\fm,p} \ol{\bb}_\eta \le 0$ in the weak sense on a neighborhood of $\eta(\R)$. 
\end{lemma}

Precisely, we assume timelike $\ez$-completeness for $\ez$ as in \eqref{eq:erange} also when $N \in [0,1]$.

\begin{proof}
It is sufficient to find a neighborhood of $\eta(0)$ on which the claim holds.
Lemma~\ref{lm:GH3.3} and Proposition~\ref{pr:GH5.1} provide a neighborhood $U$ of $\eta(0)$ such that, for sufficiently large $k \in \N$, there is a maximizing timelike geodesic from each $x \in U$ to $\eta(k)$.
Moreover, such a sequence of maximizing geodesics $\zeta_k \colon [0,\tau(x,\eta(k))] \lra M$ converges uniformly (up to  subsequences and with respect to any complete auxiliary Riemannian metric on $M$) to a timelike asymptote $\zeta$ of $\eta$.
Arguing as in  \cite[Lemma 3.4]{GH}, up to passing to another subsequence, also the  initial vectors $\dot \zeta_k(0)$ converge to $\dot\zeta(0)$, by smooth dependence of geodesics  on initial data.

Put $\tau_k(x):=-\tau(x,\eta(k))$, which is Lipschitz and hence differentiable almost everywhere on $U$ by \cite[Theorem 3.11]{BO}.
Moreover, since $\bb_{\eta,t} \searrow \bb_{\eta}$ as $t \to \infty$ and $\bb_{\eta}$ is Lipschitz on $U$ (Theorem~\ref{th:GH3.7}), $(\tau_k)_k$ is uniformly bounded on $U$.
Combining this with the equi-semiconcavity of $(\tau_k)_k$ from Lemma~\ref{lm:concave}, we infer that, after possibly shrinking $U$, the family $(\tau_k)_k$ is uniformly Lipschitz on $U$.
Moreover, by Lemma~\ref{lm:zeta}\eqref{it:zeta3}, we have
\[
\nabla(-\tau_k)(x)\to \nabla(-\bb_\eta)(x)\qquad\text{as }k\to\infty
\]
for almost every $x\in U$ at which $\bb_\eta$ and $\tau_k$ are differentiable (see also Theorem~\ref{th:GH3.7}).
Therefore, for any nonnegative test function $\phi\in C_c^\infty(U)$, the dominated convergence theorem yields
\[
\lim_{k\to\infty}\int_U \dd\phi\bigl(\rev{\nabla}\tau_k\bigr)\,\dd\fm
=
\int_U \dd\phi\bigl(\rev{\nabla}\bb_\eta\bigr)\,\dd\fm,
\]
where we also used $\rev{\nabla}\bb_\eta=-\nabla(-\bb_\eta)$.

Now, given $x \in U$, it follows from the reverse form of the $p$-d'Alembertian comparison estimate \eqref{eq:p-comp'} that
\begin{equation}\label{eq:tau_k}
\rev{\square}_{\fm,p} \tau_k (x)
 \le \e^{\frac{2(\ez-1)}{n-1}\psi_\fm (\dot{\zeta}_k(0))}
 \biggl( c\int_0^{r_k} \e^{\frac{2(\ez-1)}{n-1}\psi_\fm (\dot{\zeta}_k(s))} \,\dd s \biggr)^{-1},
\end{equation}
where $r_k:=\tau(x,\eta(k))$ and $c=c(N,\ez)>0$ as in \eqref{eq:c}.
We claim that
\begin{equation}\label{eq:tau_lim}
\limsup_{k \to \infty} \rev{\square}_{\fm,p} \tau_k (x) \le 0.
\end{equation}
This is indeed immediate from \eqref{eq:tau_k} and the future timelike $\ez$-completeness (recall Definition~\ref{df:cmplt}).
Here, for completeness, we also give a proof under \eqref{eq:cmplt}.
If \eqref{eq:tau_lim} does not hold true,
then we deduce from \eqref{eq:tau_k} that there is a subsequence of $(\tau_k)_k$,
again denoted by $(\tau_k)_k$ for brevity, such that
\[ \e^{\frac{2(\ez-1)}{n-1}\psi_\fm (\dot{\zeta}_k(0))}
 \biggl( c\int_0^{r_k} \e^{\frac{2(\ez-1)}{n-1}\psi_\fm (\dot{\zeta}_k(s))} \,\dd s \biggr)^{-1} \ge C \]
for some constant $C>0$.
By passing to the limit as $k \to \infty$ (up to a subsequence), we find an asymptote $\zeta\colon [0,\infty) \lra M$ of $\eta$ with $\zeta(0)=x$, which satisfies
\[ \e^{\frac{2(\ez-1)}{n-1}\psi_\fm (\dot{\zeta}(0))}
 \biggl( c\int_0^{\infty} \e^{\frac{2(\ez-1)}{n-1}\psi_\fm (\dot{\zeta}(s))} \,\dd s \biggr)^{-1} \ge C. \]
This contradicts \eqref{eq:cmplt}, thus we obtain \eqref{eq:tau_lim}.
We remark that, in the case of $N \in [n,\infty)$, we can take $\ez=1$ and timelike $1$-completeness always holds.

Under a uniform bound $\rev{\square}_{\fm,p} \tau_k \le C$ on $U$ guaranteed by \eqref{eq:tau_k} and the smoothness of $\psi_\fm$, Fatou's lemma and $\phi \ge 0$ imply
\[ \lim_{k \to \infty} \int_U \phi \cdot \rev{\square}_{\fm,p} \tau_k \,\dd\fm
 \le \int_U \phi \cdot \limsup_{k \to \infty} \rev{\square}_{\fm,p} \tau_k \,\dd\fm \le 0. \]
Therefore, recalling that $F^* (-\dd \bb_{\eta}) =F^*(-\dd\tau_k) =1$ $\fm$-a.e.\ on $U$ (cf.\ Lemma~\ref{lm:zeta}\eqref{it:zeta3}), we obtain
\begin{align*}
\int_U \phi \cdot \rev{\square}_{\fm,p} \bb_\eta \,\dd\fm
&= \int_U \phi \cdot \rev{\square}_\fm \bb_\eta \,\dd\fm
=-\int_U \dd\phi \bigl( \rev{\nabla} \bb_\eta \bigr) \,\dd\fm \\
&=-\lim_{k \to \infty} \int_U \dd\phi \bigl( \rev{\nabla} \tau_k \bigr) \,\dd\fm
=\lim_{k \to \infty} \int_U \phi \cdot \rev{\square}_{\fm} \tau_k \,\dd\fm \\
&= \lim_{k \to \infty} \int_U \phi \cdot \rev{\square}_{\fm,p} \tau_k \,\dd\fm \le 0,
\end{align*}
which completes the proof of $\rev{\square}_{\fm,p} \bb_\eta \le 0$.
The other inequality $\square_{\fm,p} \ol{\bb}_\eta \le 0$ is shown in the same way by using \eqref{eq:p-comp} and past timelike $\ez$-completeness.
\end{proof}

The next proposition, along the lines of \cite[Proposition~9]{quintet}, is the key step where the ellipticity of the $p$-d'Alembertian enables us to apply the maximum principle.

\begin{proposition}[$p$-harmonicity]\label{pr:max}
In the same situation as Lemma~$\ref{lm:sharm}$, there exists a neighborhood $W$ of $\eta(\R)$ on which we have $\bb_{\eta}+\ol{\bb}_{\eta}=0$ pointwise as well as $\square_{\fm,p}(-\bb_\eta) =\square_{\fm,p}\ol{\bb}_\eta =0$ in the weak sense.
In particular, $\bb_\eta$ and $\ol{\bb}_\eta$ are $C^{1,1}_{\loc}$ on $W$.
\end{proposition}

\begin{proof}
Set $u:=\bb_\eta +\ol{\bb}_\eta$ and recall from \eqref{eq:b+b} that $u \ge 0$.
Let us consider
\[ h_t:=-\bb_\eta +tu =(1-t)(-\bb_\eta) +t\ol{\bb}_\eta. \]
Note that $h_0=-\bb_\eta$, $h_1=\ol{\bb}_\eta$, and that $-h_t$ is a time function.
Thus, we observe from Lemma~\ref{lm:sharm} that $\square_{\fm,p}h_1 \le 0$ and $\square_{\fm,p}h_0 =-\rev{\square}_{\fm,p}\bb_\eta \ge 0$ in the weak sense on some neighborhood $W$ of $\eta(\R)$.

Now, we fix a local coordinate system $(x^i)_{i=1}^n$ on a (small) open set $U$ in $W$ and take a nonnegative test function $\phi \in C^\infty_c(U)$ of compact support.
Then, we have
\begin{align*}
0 &\ge \int_U \phi (\square_{\fm,p} h_1 -\square_{\fm,p}h_0) \,\dd\fm \\
&= -\int_U \dd\phi \bigl( F^*(\dd h_1)^{p-2} \cdot \nabla h_1 -F^*(\dd h_0)^{p-2} \cdot \nabla h_0 \bigr) \,\dd\fm \\
&= -\int_U \int_0^1 \frac{\dd}{\dd t} \Bigl[ F^*(\dd h_t)^{p-2} \cdot \dd\phi(\nabla h_t) \Bigr] \,\dd t \,\dd\fm.
\end{align*}
It follows from \eqref{eq:Leg}, \eqref{eq:grad} and \eqref{eq:Euler} that the right-hand side can be calculated as
\begin{align*}
&= -\int_U \int_0^1 \biggl\{ (2-p)F^*(\dd h_t)^{p-4} \dd u(\nabla h_t) \cdot \dd\phi(\nabla h_t) \\
&\qquad\qquad\qquad
+F^*(\dd h_t)^{p-2} \cdot \dd\phi\biggl( \sum_{i,j=1}^n g^*_{ij}(\dd h_t)\frac{\del u}{\del x^j} \frac{\del}{\del x^i} \biggr) \biggr\} \,\dd t \,\dd\fm \\
&= -\int_U \int_0^1 \sum_{i,j=1}^n \frac{\del\phi}{\del x^i} \frac{\del u}{\del x^j} F^*(\dd h_t)^{p-2} \\
&\qquad\qquad\qquad
\times \biggl\{ \frac{2-p}{F^*(\dd h_t)^2} \sum_{k,l=1}^n g^*_{ik}(\dd h_t) \frac{\del h_t}{\del x^k} g^*_{jl}(\dd h_t) \frac{\del h_t}{\del x^l} +g^*_{ij}(\dd h_t) \biggr\} \,\dd t \,\dd\fm.
\end{align*}
Therefore, by writing $\fm=\sigma \,\dd x^1 \cdots \dd x^n$ on $U$, $u$ is a supersolution to a linear second order differential operator, namely
\[ \sum_{i,j=1}^n \frac{\del}{\del x^i}\biggl[ a_{ij} \frac{\del u}{\del x^j} \biggr] \le 0 \]
in the weak sense on $U$ with
\[ a_{ij} :=\sigma \int_0^1 F^*(\dd h_t)^{p-2} \biggl\{ \frac{2-p}{F^*(\dd h_t)^2} \sum_{k,l=1}^n g^*_{ik}(\dd h_t) \frac{\del h_t}{\del x^k} g^*_{jl}(\dd h_t) \frac{\del h_t}{\del x^l} +g^*_{ij}(\dd h_t) \biggr\} \,\dd t. \]

In the Fermi(--Walker) coordinates along $\eta$ such that $(\del/\del x^i)_{i=1}^n$ is $g_{\nabla(-\bb_\eta)}$-orthonormal and $\del/\del x^1=\dot{\eta}$ on $\eta$, $(a_{ij})$ is the diagonal matrix $\sigma \cdot \mathrm{diag}(1-p,1,\ldots,1)$ on $\eta$, which is positive definite.
Hence, taking a smaller neighborhood $W$ of $\eta(\R)$ if necessary, $u$ is a supersolution to an elliptic equation on $W$.
Combining this with $u \ge 0$ and $u=0$ on $\eta$, we can apply the maximum principle, yielding $u=0$ on $W$.
Then, we deduce from $\bb_\eta =-\ol{\bb}_\eta$ and Lemmas~\ref{lm:concave}, \ref{lm:sharm} that $\bb_\eta$ is both semiconcave and semiconvex (hence $C^{1,1}_{\loc}$) and satisfies $\square_{\fm,p}(-\bb_\eta) =\square_{\fm,p}\ol{\bb}_\eta =0$ in the weak sense (see \cite[Proposition~9]{quintet} for more details).
\end{proof}

\subsection{Bochner-type identity and smoothness}\label{ssc:Bochner}

In order to obtain the smoothness of the Busemann function $\bb_\eta$, we generalize the Bochner identity to our setting.
Different from the strategy in  \cite{MS:22} and \cite{quintet} directly working on $H_p$ in \eqref{eq:H_p} (cf.\ related works \cite{Lee,Oham} on Hamiltonian systems), we shall closely follow the lines of thought in the Finsler case \cite{Obook,OSbw}.
In fact, we can apply the same calculation regardless of the signature of the metric; the only difference is that the quantity \eqref{eq:q-HS} corresponding to the Hilbert--Schmidt norm is not necessarily nonnegative in the Lorentzian setting.
Due to the close analogy with the positive definite setting, we only give an outline here and refer to \cite[\S 3.1]{OSbw} and \cite[\S 12.2]{Obook} for more details.

Let $U \subset M$ be an open set and take $h \in C^3(U)$ such that $-h$ is temporal.
For $y \in U$ and small $t>0$, define a map $\cT_t$ and a vector field $V_t$ by
\[ \cT_t(y) :=\exp_y\bigl( t\nabla h(y) \bigr), \qquad
V_t\bigl( \cT_t(y) \bigr) =\frac{\dd}{\dd t}\bigl[ \cT_t(y) \bigr]. \]
Fix $x \in U$ and put $\zeta(t):=\cT_t(x)$.
Since $\zeta$ is a timelike geodesic and $\dot{\zeta}(t)=V_t(\zeta(t))$, we have
\begin{equation}\label{eq:V_t}
\sum_{i=1}^n \frac{\del V_t^i}{\del t}\bigl( \zeta(t) \bigr) \frac{\del}{\del x^i}\Big|_{\zeta(t)}
 +D^{V_t}_{V_t} V_t \bigl( \zeta(t) \bigr) =0.
\end{equation}
Let $(e_i)_{i=1}^n$ be an orthonormal basis of $(T_xM,g_{\nabla h})$
such that $e_1=\nabla h(x)/F(\nabla h(x))$,
and consider vector fields along $\zeta$ given by
\[ E_i(t) :=(\dd \cT_t)_x (e_i) =\frac{\del}{\del s}\Bigl[ \cT_t\bigl( \exp_x(se_i) \bigr) \Bigr]_{s=0}. \]
By the latter expression, $E_i$ is a Jacobi field with $E_i(0)=e_i$.
Since $(E_i(t))_{i=1}^n$ is a basis of $T_{\zeta(t)}M$ (for small $t>0$), we can introduce an $n \times n$ matrix $B(t)=(b_{ij}(t))$ by $D^{\dot{\zeta}}_{\dot{\zeta}} E_i(t)=\sum_{j=1}^n b_{ij}(t) E_j(t)$.
Then, we have the following \emph{Riccati equation}:
\begin{equation}\label{eq:Riccati}
\bigl[ \trace B \bigr]'(t) +\trace\bigl[ B(t)^2 \bigr] +\Ric\bigl( \dot{\zeta}(t) \bigr) =0.
\end{equation}

Observe that $B(t)$ represents the covariant derivative of $V_t$ in the sense that
\[ D^{\dot{\zeta}(t)}_{E_i(t)} V_t =\sum_{j=1}^n b_{ij}(t)E_j(t). \]
In particular, since $V_0=\nabla h$, we have $b_{ij}(0) =g_{\nabla h}(\nabla^2_{e_i}h,e_j) g_{\nabla h}(e_j,e_j)$ and infer from the symmetry \eqref{eq:symm} that
\begin{equation}\label{eq:q-HS}
\trace\bigl[ B(0)^2 \bigr]
=\sum_{i,j=1}^n g_{\nabla h}(\nabla^2_{e_i} h,e_j)^2 g_{\nabla h}(e_i,e_i) g_{\nabla h}(e_j,e_j)
=: \HS_{\nabla h}(\nabla^2 h),
\end{equation}
which corresponds to the square of the Hilbert--Schmidt norm of $\nabla^2 h$ with respect to $g_{\nabla h}$ in the positive definite case, but is not necessarily nonnegative in the current setting.

Concerning the first term in \eqref{eq:Riccati}, we find
\[ \trace \bigl[ B(t) \bigr] =\div_\fm V_t\bigl( \zeta(t) \bigr) +(\psi_\fm \circ \dot{\zeta})'(t), \]
and then \eqref{eq:V_t} (together with the linearity of $\div_\fm$ from \eqref{eq:div}) implies
\[ \frac{\dd}{\dd t} \Bigl[ \div_\fm V_t\bigl( \zeta(t) \bigr) \Bigr]
=\dd(\div_\fm V_t)\bigl( \dot{\zeta}(t) \bigr) -\div_\fm (D^{V_t}_{V_t} V_t)\bigl( \zeta(t) \bigr). \]
Now, at $t=0$, we observe $\div_\fm V_0 =\square_\fm h$ and, since
\[ g_{\nabla h}\biggl( D^{\nabla h}_{\nabla h} \bigl[ \nabla h \bigr],\frac{\del}{\del x^i} \biggr)
=g_{\nabla h}\Bigl( D^{\nabla h}_{\del /\del x^i} \bigl[ \nabla h \bigr],\nabla h \Bigr)
=\frac{1}{2} \frac{\del}{\del x^i} \Bigl[ g_{\nabla h}(\nabla h,\nabla h) \Bigr], \]
it follows that
\[ D^{\nabla h}_{\nabla h} \bigl[ \nabla h \bigr]
= \sum_{i,j=1}^n g^{ij}(\nabla h) g_{\nabla h}\biggl( D^{\nabla h}_{\nabla h} \bigl[ \nabla h \bigr],\frac{\del}{\del x^i} \biggr) \frac{\del}{\del x^j}
= \nabla^{\nabla h} \biggl[ -\frac{F(\nabla h)^2}{2} \biggr], \]
where we define
\[ \nabla^{\nabla h} u :=\sum_{i,j=1}^n g^{ij}(\nabla h) \frac{\del u}{\del x^j} \frac{\del }{\del x^i} \]
(we do not impose any causality condition on $u$).
Plugging those into \eqref{eq:Riccati} at $t=0$, we arrive at the following.

\begin{proposition}[Bochner-type identity]\label{pr:Boch}
Let $(M,L,\fm)$ be a weighted Finsler spacetime.
For $h \in C^3(U)$ on an open set $U \subset M$ such that $-h$ is temporal, we have
\begin{equation}\label{eq:Boch}
\div_\fm \biggl( \nabla^{\nabla h} \biggl[ \frac{F(\nabla h)^2}{2} \biggr] \biggr)
+\dd(\square_\fm h)(\nabla h) +\Ric_\infty (\nabla h) +\HS_{\nabla h}(\nabla^2 h) =0
\end{equation}
pointwise on $U$.
\end{proposition}

We remark that, if we replace $F(\nabla h)^2$ with $-g_{\nabla h}(\nabla h,\nabla h)$, then \eqref{eq:Boch} is exactly the same form as the Bochner identity in the positive definite case in \cite{Obook,OSbw}.

\begin{remark}[Comparison with Raychaudhuri equation]\label{rm:Raych}
The \emph{Raychaudhuri equation} can be shown by a similar argument (see, e.g., \cite{LMO1}),
however, in that we deal with an endomorphism on the $(n-1)$-dimensional subspace $N_{\zeta}(t) \subset T_{\zeta(t)}M$ which is $g_{\dot{\zeta}}$-orthogonal to $\dot{\zeta}$.
In the above discussion, the corresponding endomorphism $v \longmapsto D^{\dot{\zeta}(t)}_v V_t$ on $T_{\zeta(t)}M$ does not necessarily satisfy this condition.
A typical example where one can apply the Raychaudhuri equation is the time separation function from a point, and then the d'Alembertian comparison theorem follows (see \cite{LMO2}).
\end{remark}

Since \eqref{eq:Boch} is concerned with $\Ric_\infty$, for a weaker bound $\Ric_N \ge 0$ with $N \le 0$, we need a variant along the lines of  \cite[Lemma~3.1]{Wy}.

\begin{lemma}[$\Ric_0 \ge 0$ case]\label{lm:Wylie}
Let $(M,L,\fm)$ be a weighted Finsler spacetime with $\Ric_0 \ge 0$ in timelike directions.
Then, for $h \in C^3(U)$ as in Proposition~$\ref{pr:Boch}$, we have
\begin{equation}\label{eq:Wylie}
w\div_\fm \biggl( \nabla^{\nabla h} \biggl[ \frac{F(\nabla h)^2}{2} \biggr] \biggr)
+\dd(w\square_\fm h)(\nabla h) +w \frac{(\square_\fm h)^2}{n}
\le \frac{w}{2F(\nabla h)^2} \sum_{\alpha=2}^n \Bigl( \dd\bigl[ F(\nabla h)^2 \bigr](e_\alpha) \Bigr)^2
\end{equation}
at any $x \in U$, where $w(\zeta(t)):=\e^{2\psi_\fm(\dot{\zeta}(t))/n}$ for $\zeta(t):=\exp_x(t\nabla h(x))$ and $(e_i)_{i=1}^n$ is an orthonormal basis of $(T_xM,g_{\nabla h})$ such that $e_1=\nabla h(x)/F(\nabla h(x))$.
Moreover, in the case where $F(\nabla h)$ is constant, equality holds if and only if $\Ric_0(\nabla h(x))=0$ and $\nabla^2_v h=cv$ for some $c \in \R$ and all $v \in T_xM$.
\end{lemma}

\begin{proof}
We first observe
\[ \dd(w\square_\fm h)(\nabla h)
=w \cdot \dd(\square_\fm h)(\nabla h) +2w\frac{(\psi_\fm \circ \dot{\zeta})'(0)}{n} \cdot \square_\fm h.\]
It follows from \eqref{eq:q-HS}, the Cauchy--Schwarz inequality and \eqref{eq:g_X} that, at $x$,
\begin{align*}
\HS_{\nabla h}(\nabla^2 h)
&\ge \sum_{i=1}^n g_{\nabla h}(\nabla^2_{e_i} h,e_i)^2
-2\sum_{\alpha=2}^n g_{\nabla h}(\nabla^2_{e_\alpha} h,e_1)^2 \\
&\ge \frac{1}{n} \Biggl( -g_{\nabla h}(\nabla^2_{e_1} h,e_1) +\sum_{i=2}^n g_{\nabla h}(\nabla^2_{e_i} h,e_i) \Biggr)^2
-\frac{2}{F(\nabla h)^2} \sum_{\alpha=2}^n g_{\nabla h}(\nabla^2_{e_\alpha} h,\nabla h)^2 \\
&=\frac{(\square h)^2}{n} -\frac{1}{2F(\nabla h)^2} \sum_{\alpha=2}^n \Bigl( \dd\bigl[ F(\nabla h)^2 \bigr](e_\alpha) \Bigr)^2.
\end{align*}
Moreover, we have
\[ (\square h)^2 =\bigl( \square_\fm h +(\psi_\fm \circ \dot{\zeta})'(0) \bigr)^2
= (\square_\fm h)^2 +2(\psi_\fm \circ \dot{\zeta})'(0) \square_\fm h +(\psi_\fm \circ \dot{\zeta})'(0)^2. \]
Comparing these with \eqref{eq:Boch}, we obtain
\begin{align*}
&\dd(w\square_\fm h)(\nabla h) +w \frac{(\square_\fm h)^2}{n} -\frac{w}{2F(\nabla h)^2} \sum_{\alpha=2}^n \Bigl( \dd\bigl[ F(\nabla h)^2 \bigr](e_\alpha) \Bigr)^2 \\
&\le w \cdot \dd(\square_\fm h)(\nabla h) +w\HS_{\nabla h}(\nabla^2 h) -w\frac{(\psi_\fm \circ \dot{\zeta})'(0)^2}{n} \\
&= -w\div_\fm \biggl( \nabla^{\nabla h} \biggl[ \frac{F(\nabla h)^2}{2} \biggr] \biggr) -w\Ric_\infty(\nabla h) -w\frac{(\psi_\fm \circ \dot{\zeta})'(0)^2}{n} \\
&= -w\div_\fm \biggl( \nabla^{\nabla h} \biggl[ \frac{F(\nabla h)^2}{2} \biggr] \biggr) -w\Ric_0(\nabla h).
\end{align*}
Since $\Ric_0 \ge 0$ by assumption, this completes the proof of \eqref{eq:Wylie}.

When $F(\nabla h)$ is constant, the right-hand side of \eqref{eq:Wylie} vanishes.
Then, by the above proof, equality holds if and only if $\Ric_0(\nabla h(x))=0$ and $g_{\nabla h}(\nabla^2_{e_i} h,e_j)=c \cdot g_{\nabla h}(e_i,e_j)$ for some $c \in \R$, which means $\nabla^2_v h=cv$.
\end{proof}

With \eqref{eq:Boch} and \eqref{eq:Wylie} in hand, we can show the smoothness of $\bb_\eta$.
We remark that the $N \in [0,1]$ case is excluded here.

\begin{corollary}[Smoothness]\label{cr:smooth}
Let $(M,L,\fm)$ be a timelike geodesically complete, weighted Finsler spacetime as in Theorem~$\ref{th:LFsplit}$.
Then, the Busemann function $\bb_\eta$ is $C^\infty$ and satisfies
\begin{equation}\label{eq:H=0}
\nabla^2(-\bb_\eta) \equiv 0
\end{equation}
on a neighborhood $W$ of $\eta(\R)$.
\end{corollary}

\begin{proof}
We put $h:=-\bb_\eta$ for brevity.
Recall from Lemma~\ref{lm:zeta} and Proposition~\ref{pr:max} that $h$ is $C^{1,1}_{\loc}$, $F(\nabla h) \equiv 1$ on $W$, and $\square_{\fm,p}h =\square_\fm h =0$ almost everywhere on $W$.

First, suppose that $h$ is $C^3$.
When $N \in [n,\infty]$, we deduce from \eqref{eq:Boch} and the hypothesis $\Ric_N \ge 0$ that 
\[ \HS_{\nabla h}(\nabla^2 h)
=-\Ric_\infty (\nabla h)
\le -\Ric_N (\nabla h) \le 0. \]
Thanks to $F(\nabla h) \equiv 1$, in \eqref{eq:q-HS}, for $\alpha=2,\ldots,n$, we observe
\[ g_{\nabla h}(\nabla^2_{e_1} h,e_\alpha) =g_{\nabla h}(\nabla^2_{e_\alpha} h,e_1)
=-\frac{1}{2} \dd\bigl[ F(\nabla h)^2 \bigr](e_\alpha) =0. \]
This yields $g_{\nabla h}(\nabla^2_{e_i} h,e_j) =0$ for all $i,j$, and hence \eqref{eq:H=0} holds and $h$ is $C^\infty$.

If $N<0$, then $\Ric_0 \ge \Ric_N \ge 0$ and we can apply Lemma~\ref{lm:Wylie}.
Since equality holds in \eqref{eq:Wylie} with $F(\nabla h) \equiv 1$, we have $\Ric_0(\nabla h)=0$ and $\nabla^2_v h=cv$ for some $c \in \R$.
Now, combining $\Ric_0(\nabla h)=0$ with the hypothesis $\Ric_N \ge 0$, we find that
\begin{align*}
\Ric_N \bigl( \nabla h(x) \bigr)
&=\Ric_0 \bigl( \nabla h(x) \bigr) -\biggl( \frac{1}{n}+\frac{1}{N-n} \biggr) (\psi_\fm \circ \dot{\zeta})'(0)^2 \\
&=-\frac{N}{n(N-n)} (\psi_\fm \circ \dot{\zeta})'(0)^2
\end{align*}
is nonnegative, where $\zeta$ is the geodesic with $\dot{\zeta}(0)=\nabla h(x)$.
Since $N<0$, this implies $(\psi_\fm \circ \dot{\zeta})'(0)=0$, thereby
\[ \square h =\square_\fm h +(\psi_\fm \circ \dot{\zeta})'(0) =0. \]
Therefore, we obtain $c=0$ and the proof is completed.

For general $\bb_\eta \in C^{1,1}_{\loc}$, we approximate it with $C^\infty$-functions, and obtain \eqref{eq:H=0} almost everywhere by the above argument.
This yields that \eqref{eq:H=0} holds everywhere and $\bb_\eta$ is $C^\infty$.
We refer to \cite[Corollary~14]{quintet} for further details of the approximation procedure, via the integrated form of \eqref{eq:Boch}:
\[ \int_M \biggl\{ \dd\phi \biggl( \nabla^{\nabla h} \biggl[ \frac{F(\nabla h)^2}{2} \biggr] \biggr)
 +\square_\fm h \cdot \div_\fm(\phi \nabla h) \biggr\} \,\dd\fm
 =\int_M \phi \bigl\{ \Ric_\infty (\nabla h) +\HS_{\nabla h}(\nabla^2 h) \bigr\} \,\dd\fm \]
for (nonnegative) test functions $\phi \in C^{\infty}_c(M)$ of compact support. To sketch the argument, one may approximate $h$ with smooth $h_\varepsilon$ which satisfies $F(\nabla h_\varepsilon) \geq 1- \varepsilon$ and $h_\varepsilon \to h$ in $W^{2,2}_{\mathrm{loc}}(W)$, as well as $\nabla^2 h_\varepsilon \to \nabla^2 h$ (and hence $\square_{\mathfrak m} h_\varepsilon \to \square_{\mathfrak m} h$) $\mathfrak m$-a.e. Using approximate versions of the arguments given above for $h$, now for $h_\varepsilon$, in the integrated Bochner-type formula above, one may thereupon pass to the limit to ultimately obtain $\nabla^2h = 0$ $\mathfrak m$-a.e., which, in coordinates, is easily seen to imply that $h$ is smooth. 
\end{proof}

\begin{remark}[{$N \in [0,1]$} case] \label{rm:Wylie}
Wylie's \cite[Lemma~3.1]{Wy} is in fact more general and concerned with the situation that $\nabla^2 h$ has at most $k$ non-zero eigenvalues under the assumption $\Ric_{n-k} \ge 0$.
Applying this with $k=n-1$ for a Busemann function, whose second order derivative vanishes along asymptotes, we could extend the splitting theorem to $N<1$ (cf. \cite[Corollary~1.3]{Wy},  \cite[Theorem~1.5(i)]{WW}), and we can even obtain a warped product splitting for $N=1$ (\cite[Theorem~1.2]{Wy}, \cite[Theorem~1.5(ii)]{WW}).
In the above proof, however, $C^\infty$-functions $h_{\varepsilon}$ approximating $\bb_\eta$ may have $n$ non-zero eigenvalues; this is the reason why we restricted ourselves to $N<0$.
\end{remark}

\section{Splitting theorems}\label{sc:split}

We are in a position to complete the proofs of the splitting theorems.
We first consider the timelike geodesically complete (TGC) case, and the globally hyperbolic (GH) case is deferred to Subsection~\ref{ssc:ghyp}.

\subsection{Diffeomorphic splitting}\label{ssc:prf1}

\begin{proof}[Proof of Theorem~$\ref{th:LFsplit}$ $($TGC case$)$]
We deduce from Lemmas~\ref{lm:zeta}, \ref{lm:line} together with Corollary~\ref{cr:smooth} that, for every point $x \in W$, the unique straight line bi-asymptotic to $\eta$ passing through $x$ is given by the integral curve of $\nabla(-\bb_\eta)$.
In other words, every integral curve of $\nabla(-\bb_\eta)$ is a straight line bi-asymptotic to $\eta$.
This enables us to apply standard arguments in splitting theorems in Lorentzian manifolds (one can alternatively apply Wu's de Rham-type decomposition theorem \cite{Wu}).
Precisely, since $\nabla^2(-\bb_\eta)=0$ on a neighborhood $W$ of $\eta(\R)$, also the $g_{\nabla(-\bb_\eta)}$-Hessian of $-\bb_\eta$ vanishes by \cite[Proposition~3.2]{LMO1}, thus we deduce that $\nabla(-\bb_\eta)$ is a parallel (and thus a complete timelike Killing) vector field for $g_{\nabla(-\bb_\eta)}$.
Hence, taking into account that $\nabla(-\bb_\eta)$ is also the gradient of $-\bb_\eta$ with respect to $g_{\nabla(-\bb_\eta)}$, we set  
$\Sigma':=\bb_\eta^{-1}(0) \cap W$ and 
\begin{equation}\label{eq:Theta}
\Theta\colon  \R \times \Sigma'  \lra  M,\quad \quad \Theta(t,x):=\zeta_x(t),
\end{equation}
where  $\zeta_x$ is the integral curve of $\nabla(-\bb_\eta)$  with $\zeta_x(0)=x$.  The map $\Theta$ is then an isometry from $(\R \times \Sigma',-\dd t^2 +g_{\nabla(-\bb_\eta)}|_{T\Sigma'})$  onto its image in $(M,g_{\nabla(-\bb_\eta)})$. 
Moreover, we deduce from $(\psi_\fm \circ \dot{\zeta}_x)'=0$ shown in the proof of Corollary~\ref{cr:smooth} that the measure $\fm$ is also decomposed into the product $\dd t \times \fn$.
Finally, we can show the global splitting by the standard open-and-closed argument (see, e.g., \cite[Theorem~7.3]{LMO3}).
This completes the proof of Theorem~\ref{th:LFsplit}.
\end{proof}

We remark that $\Sigma$ is given as the level surface $\bb_\eta^{-1}(0)$ equipped with the Riemannian metric $g_{\nabla(-\bb_\eta)}|_{T\Sigma}$.
As for $L|_{T\Sigma}$, one can say the following.

\begin{corollary}[Causal character of $\Sigma$]\label{cr:Sigma}
In the situation of Theorem~$\ref{th:LFsplit}$, the level surface $\Sigma=\bb_{\eta}^{-1}(0)$ is transversal to future-directed causal vectors in $TM$.
Moreover, if $L$ is reversible and $\dim M \ge 3$, then $L(v)>0$ for any $v \in T\Sigma \setminus \{0\}$.
\end{corollary}

\begin{proof}
If $v \in T_x\Sigma \setminus \{0\}$ is future-directed causal, then we deduce from the reverse Cauchy--Schwarz inequality \eqref{eq:rCS} that
\[ 0=\dd\bb_\eta(v) \ge \sqrt{4L^*\bigl( -\dd\bb_\eta(x)\bigr) L(v)} \ge 0. \]
Recalling the equality condition in \eqref{eq:rCS}, we find that $L(\nabla(-\bb_\eta)(x))=L(v)=0$, which contradicts $F(\nabla(-\bb_\eta))=1$.

In the case where $L$ is reversible and $\dim M \ge 3$, the set $\Omega'_x \subset T_xM$ of timelike vectors has exactly two connected components, corresponding to the future and past directions (recall Remark~\ref{rm:Omega}).\ In this case, since $-v \in T_x\Sigma$ is not future-directed causal as well, every $v \in T_x\Sigma \setminus \{0\}$ satisfies $L(v)>0$.
\end{proof}

We remark that, even when $L>0$ on $T_x\Sigma \setminus \{0\}$, we do not know if $L|_{T_x\Sigma}$ is strongly convex (in the sense that the Hessian of $L|_{T_x\Sigma \setminus \{0\}}$ is positive definite).

\subsection{Isometric translations in Berwald spacetimes}\label{ssc:prf2}

Next, we prove Theorem~\ref{th:Berwald} concerning Berwald spacetimes.
The proof is common to the TGC and GH cases (once the splitting as in Theorem~\ref{th:LFsplit} is established).
Since the covariant derivative is independent of a reference vector by the very definition of Berwald spacetimes (recall Definition~\ref{df:Ber}), we may omit the reference vector in the following.

\begin{lemma}[Busemann function is affine]\label{lm:affine}
Let $(M,L,\fm)$ be a weighted Berwald spacetime as in Theorem~$\ref{th:Berwald}$.
Then, for any geodesic $\xi\colon [0,1] \lra M$, we have $(\bb_\eta \circ \xi)'' \equiv 0$.
In particular, for each $t \in \R$, $\bb_{\eta}^{-1}(t)$ is a totally geodesic hypersurface with respect to $L$.
\end{lemma}

\begin{proof}
This is straightforward from \eqref{eq:g_X} and \eqref{eq:H=0}:
\begin{align*}
(\bb_\eta \circ \xi)''
&= \bigl[ -g_{\nabla(-\bb_\eta)} \bigl( \nabla(-\bb_\eta)(\xi),\dot{\xi} \bigr) \bigl]' \\
&= -g_{\nabla(-\bb_\eta)}\bigl( D_{\dot{\xi}}\bigl (\nabla(-\bb_\eta) \bigr),\dot{\xi} \bigr) -g_{\nabla(-\bb_\eta)}\bigl( \nabla(-\bb_\eta)(\xi),D_{\dot{\xi}} \dot{\xi} \bigr) =0.
\end{align*}
Precisely, the first term vanishes in general thanks to \eqref{eq:H=0}, while the second term vanishes only in the Berwald case (since $D^{\nabla(-\bb_\eta)}_{\dot{\xi}} \dot{\xi} \neq 0$ in general).
\end{proof}

To show Theorem~\ref{th:Berwald}, we can follow the lines of  \cite[Proposition~5.2]{Osplit}.

\begin{proof}[Proof of Theorem~$\ref{th:Berwald}$]
\eqref{it:isom}
Take $v \in T_xM \setminus \{0\}$ and put $V(t):=\dd\Phi_t(v)$.
It follows from Theorem~\ref{th:LFsplit} that $V$ is a parallel vector field with respect to $g_{\nabla(-\bb_\eta)}$ along the straight line $\zeta(t):=\Phi_t(x)$.
Hence, we have
\[ \frac{\dd}{\dd t}L(V)
=\frac{1}{2} \frac{\dd}{\dd t} \bigl[ g_V(V,V) \bigr]
=g_V(D_{\dot{\zeta}} V,V) \equiv 0, \]
thereby $\Phi_t$ is isometric.
More precisely, we observe from \eqref{eq:g_X} that
\[ \frac{\dd}{\dd t} \bigl[ g_V(V,V) \bigr]
=2g_V(D^V_{\dot{\zeta}}V,V), \]
and the Berwald condition and  \cite[Proposition~3.2]{LMO1} imply $D^V_{\dot{\zeta}} V =D^{\dot{\zeta}}_{\dot{\zeta}}V =D^{g_{\nabla(-\bb_\eta)}}_{\dot{\zeta}} V \equiv 0$, where $D^{g_{\nabla(-\bb_\eta)}}$ denotes the covariant derivative with respect to $g_{\nabla(-\bb_\eta)}$.

\eqref{it:geod}
We shall split the geodesic equation $D_{\dot{\xi}} \dot{\xi} =0$ on $M$ into those on $\R$ and $\Sigma$.
Given a local coordinate system $(z^\alpha)_{\alpha=1}^{n-1}$ on an open set $U \subset \Sigma$, we consider a coordinate system $(x^i)_{i=1}^n$ on $\R \times U \subset M$ given by $x^1=\bb_\eta$ and $x^i=z^{i-1}$ for $i \ge 2$.
Then, we have
\[ \frac{\del g_{ij}}{\del x^1}\bigl( \nabla(-\bb_\eta) \bigr) =0
\quad \text{for}\ i,j=1,\ldots,n \]
by Theorem~\ref{th:LFsplit}, and
\[ \frac{\del g_{1j}}{\del x^i}\bigl( \nabla(-\bb_\eta) \bigr) =0
\quad \text{for}\ i=1,\ldots,n \]
by $F^*(-\dd\bb_\eta)=1$ for $j=1$, and by $g_{\nabla(-\bb_\eta)}(\nabla(-\bb_\eta),T\Sigma)=0$ for $j \ge 2$.
Therefore, we find $\gamma^i_{jk}(\nabla(-\bb_\eta))=0$ unless $i,j,k \ge 2$.

Moreover, by the expression \eqref{eq:N}:
\[ N^i_j(v) =\sum_{k=1}^n \gamma^i_{jk}(v) v^k -\frac{1}{2} \sum_{k,l,a,b=1}^n g^{ia}(v) \frac{\del g_{ab}}{\del v^j}(v) \gamma^b_{kl}(v) v^k v^l, \]
we deduce that $N^i_j(\nabla(-\bb_\eta))=0$ for all $i,j=1,\ldots,n$.
This implies that, by recalling the definition \eqref{eq:Gamma} of $\Gamma^i_{jk}$, $\Gamma^i_{jk}(\nabla(-\bb_\eta))=0$ holds unless $i,j,k \ge 2$.
Therefore, the geodesic equation in $M$ is written as
\[ D_{\dot{\xi}} \dot{\xi}
=\ddot{\xi}^1 \frac{\del}{\del x^1}\Big|_\xi
+\sum_{i=2}^n \biggl\{ \ddot{\xi}^i +\sum_{j,k=2}^n \Gamma^i_{jk}(\xi) \dot{\xi}^j \dot{\xi}^k \biggr\} \frac{\del}{\del x^i}\Big|_\xi =0. \]
This completes the proof.
\end{proof}

\begin{remark}[When $L|_{T\Sigma\setminus \{0\}}$ is positive]\label{rm:Sigma}
Although $\Sigma=\bb_\eta^{-1}(0)$ may not be spacelike in $L$ in general, we notice that, by the fact that parallel transports preserve $L$ and the decomposition of $\Gamma^i_{jk}$ in the proof of Theorem~\ref{th:Berwald}\eqref{it:geod}, if there exists some $x \in \bb_\eta^{-1}(0)$ such that $T_x\Sigma$ is spacelike (i.e., $L(v)>0$ for all $v\in T_x \Sigma\setminus\{0\}$) and $g_v|_{T_x\Sigma}$ is positive semi-definite for all $v \in T_x\Sigma \setminus \{0\}$, then $(\Sigma,L|_{T\Sigma})$ is a (non-strongly convex) Finsler manifold of Berwald type.
In this case, one can also give a Lorentzian metrization of $(M,L)$ (recall Remark~\ref{rm:met}).
In fact, by \cite[Remark(A), p.~2157]{MatTro}, which extends Szab\'o's metrization result of Berwald spaces \cite{Sz} to  non-strongly convex (and also non-smooth) Berwald metrics,
we obtain a Riemannian metric $h$ on $\Sigma$ such that the Levi-Civita connection of $h$ coincides with the Chern connection of $L|_{T\Sigma}$.
Then, the Levi-Civita connection of the Lorentzian metric $-\dd t^2 +h$ coincides with the Chern connection of $L$.
We also remark that, by \cite[Proposition~2, Lemma~2]{Min-cone}, if $n\geq 3$ and $L|_{T\Sigma\setminus \{0\}}>0$ , then, for all $x\in M$, $\Omega'_x$ has either one or two connected components.
\end{remark}

\subsection{Globally hyperbolic case}\label{ssc:ghyp}

We conclude with the proof of the globally hyperbolic case.

\begin{definition}[Global hyperbolicity]\label{df:gh}
We say that a Finsler spacetime $(M,L)$ is \emph{globally hyperbolic} if it is \emph{causal} (there is no closed future-directed causal curve) and all causal diamonds $J^+(x) \cap J^-(y)$, $x \le y$, are compact.
\end{definition}

We remark that, since the cone structure $\overline{\Omega}$ is proper in the sense of  \cite[Definition~2.3]{Min-causality}, according to \cite[Corollary~2.4]{Min-causality}, the above definition is equivalent to requiring the strong causality and the compactness of causal diamonds (compare also with \cite[Appendix~B]{CCGP}).

\begin{proof}[Proof of Theorem~$\ref{th:LFsplit}$ $($GH case$)$]
Let $(M,L,\fm)$ be a globally hyperbolic weighted Berwald spacetime satisfying the hypotheses in Theorem~\ref{th:LFsplit}.
In this case, $\tau$ is finite and continuous and, for any $x,y \in M$ with $x<y$, there is a maximizing geodesic from $x$ to $y$ (see  \cite[Propositions~6.8, 6.9]{Min-Ray}).

In the globally hyperbolic case, an asymptote $\zeta\colon[0,b) \lra M$ constructed as in Definition~\ref{df:coray} may not be future complete, namely $\zeta$ may be future inextendible with $b<\infty$.
Nonetheless, concerning the behavior of Busemann functions, Lemma~\ref{lm:zeta} holds for $t \in (0,b)$, and it is shown along the lines of  \cite[Lemma~3.3]{Eschenburg:1988} (see also \cite[Theorem~2]{quintet}) that $\bb_\eta$ is locally Lipschitz on a neighborhood of $\eta(\R)$ (as in Theorem~\ref{th:GH3.7}).

Furthermore, following the lines of  \cite[\S 6]{quintet}, we can apply the calculations as in Section~\ref{sc:harm}.
(In this case, due to the possibility of $b<\infty$ as above, we need the timelike $\ez$-completeness as in Definition~\ref{df:cmplt} rather than \eqref{eq:cmplt}.)
Eventually, on a neighborhood $W$ of $\eta(\R)$, we have $\bb_\eta+\ol{\bb}_\eta \equiv 0$, $\bb_\eta \in C^{\infty}(W)$, and $\nabla^2(-\bb_\eta) \equiv 0$.
Therefore, we obtain a local splitting of $W$ as in Theorems~\ref{th:LFsplit}, \ref{th:Berwald}.

Now, the difficulty in the globally hyperbolic case is that $W$ may not be represented simply as the product $\R \times (\bb_\eta^{-1}(0) \cap W)$.
To show the completeness of asymptotes, we follow the argument in \cite{Galloway:89, quintet}; at this step we need the Berwald condition.
As in \eqref{eq:Theta}, we define $\Theta(t,x):=\zeta_x(t)$ for $x \in \bb_\eta^{-1}(0) \cap W$ and $t \in (a_x,b_x)$, namely $\zeta_x\colon (a_x,b_x) \lra M$ is the future and past inextendible timelike asymptote of $\eta$ with $a_x<0<b_x$.
We shall show that $a_x=-\infty$ (i.e., $\zeta_x$ is past complete).

Let $\Sigma':=\exp_{\eta(0)}(B)$ for a small ball $B \subset T_{\eta(0)}M \cap (\dd\bb_\eta)^{-1}(0)$.
Precisely, $B$ is an open ball with respect to the restriction of $g_{\nabla(-\bb_\eta)}$ to $(\dd\bb_\eta)^{-1}(0)$, which is a Riemannian metric, and the exponential map is for $L$.
By the decomposition of geodesics as in Theorem~\ref{th:Berwald}\eqref{it:geod},
$\Sigma'$ is included in $\bb_\eta^{-1}(0)$.
We may assume that, by taking $B$ smaller if necessary, $\dot{\eta}(0)+v \in \Omega_{\eta(0)}$ for any $v \in B$.
Fix $v \in B$, put $\sigma(s):=\exp_{\eta(0)}(sv)$ for $s \in [0,1]$, and consider $\zeta_{\sigma(s)}(t)=\Theta(t,\sigma(s))$.
To show $a_{\sigma(s)}=-\infty$ for all $s \in [0,1]$, we claim that, for any fixed $T>0$,
\begin{equation}\label{eq:a_s}
a_{\sigma(s)} <-T+s
\end{equation}
holds.
To this end, denote by $I_T$ the set of $R \in [0,1]$ such that \eqref{eq:a_s} holds for all $s \in [0,R]$.
Then, clearly $0 \in I_T$ since $a_{\eta(0)}=-\infty$, and $I_T$ is open in $[0,1]$.

We prove that $R:=\sup I_T \in I_T$.
If not, then we have $a_{\sigma(s)} <-T+s<-T+R \le a_{\sigma(R)}$ for all $s \in [0,R)$.
Hence, $\zeta_{\sigma(s)}(t)$ is well-defined for all $s \in [0,R)$ and $t \in (a_{\sigma(R)},0]$.
Now, we consider a timelike geodesic
\[ \xi(s):=\exp_{\eta(-T)}\Bigl( s \bigl( \dot{\eta}(-T)+v \bigr) \Bigr) =\Theta\bigl( -T+s,\sigma(s) \bigr), \]
where $v \in T_{\eta(-T)}M$ is identified with $v \in B$ above and $\dot{\eta}(-T)+v \in \Omega_{\eta(-T)}$ by the isometry of the translation as in Theorem~\ref{th:Berwald}\eqref{it:isom}, and equality is seen from Theorem~\ref{th:Berwald}\eqref{it:geod}.
By construction, for all $s \in [0,R)$ and $t \in (a_{\sigma(R)},0]$, we find
\[ \eta(-T) \ll \xi(s) =\Theta\bigl( -T+s,\sigma(s) \bigr) \ll \Theta\bigl( a_{\sigma(R)},\sigma(s) \bigr) \ll \Theta\bigl( t,\sigma(s) \bigr). \]
Letting $s \to R$, we deduce from the closedness of causal diamonds (by the global hyperbolicity) that $\Theta(t,\sigma(R)) \in J^+(\eta(-T)) \cap J^-(\sigma(R))$ for all $t \in (a_{\sigma(R)},0]$.
This, however, contradicts the non-total imprisonment, which follows from the global hyperbolicity (see, e.g.,  \cite[Definition~2.20]{Min-causality}).
Therefore, we obtain $R \in I_T$ and $I_T=[0,1]$, which implies $a_{\sigma(s)}=-\infty$ for all $s \in [0,1]$ since $T>0$ was arbitrary in \eqref{eq:a_s}.
This completes the proof of the past completeness of asymptotes, and the future completeness can be shown in the same way as the past completeness for the reverse structure $\rev{L}$.

This concludes the proof of the fact that the local splitting neighborhood $W$ can indeed be represented as a cylindrical product of the form $\R \times (\bb_{\eta}^{-1}(0) \cap W)$.
In the globalization of the local splitting, one needs to extend the splitting of $W$ to its boundary, where one must ensure that also the boundary asymptote is complete in the globally hyperbolic case.
Using the Berwald condition, the argument is entirely analogous to the completeness of asymptotes in the local case given above, and will be omitted.
We refer to  \cite{Galloway:89} and  \cite{quintet} for details in the Lorentzian globally hyperbolic case as well as to  \cite{LMO3} for the globalization argument in the Lorentz--Finsler case.
\end{proof}

\subsection*{Acknowledgments}
EC is partially supported  by ``INdAM - GNAMPA Project''  CUP E53C25002010001.
AO is supported in part by the \"OAW-DOC scholarship of the Austrian Academy of Sciences. He would like to thank Darko Mitrovic for helpful discussions.
SO is supported by JSPS Grant-in-Aid for Scientific Research (KAKENHI) 22H04942, 24K00523.
A large part of this work was carried out during his visit to Politecnico di Bari, he is grateful for the kind hospitality.
This research was funded in part by   MUR under the Programme ``Department of Excellence'' Legge 232/2016  (Grant No. CUP - D93C23000100001) and by the Austrian Science Fund (FWF) [Grant DOI \href{https://doi.org/10.55776/PAT1996423}{10.55776/PAT1996423}, \href{https://doi.org/10.55776/P33594}{10.55776/P33594}, \href{https://doi.org/10.55776/EFP6}{10.55776/EFP6} and \href{https://doi.org/10.55776/J4913}{10.55776/J4913}]. For open access purposes, the authors have applied a CC BY public copyright license to any author accepted manuscript version arising from this submission.

\end{document}